\documentclass[12pt]{amsart}

\usepackage[a4paper]{geometry}
\usepackage{amsthm}
\usepackage{amsmath}
\usepackage{amsfonts}
\usepackage{amssymb}
\usepackage{cancel} 
\usepackage{enumerate}
\usepackage{xcolor}
\usepackage{cite}
\usepackage{setspace}

\newtheorem{theorem}{Theorem}
\numberwithin{theorem}{section}

\newtheorem{lemma}[theorem]{Lemma}
\newtheorem*{lemma*}{Lemma}
\newtheorem{corollary}[theorem]{Corollary}
\theoremstyle{definition}
\newtheorem{definition}[theorem]{Definition}
\newtheorem{roadmap}[theorem]{Roadmap}
\newtheorem{example}[theorem]{Example}
\theoremstyle{remark}

\newcommand{\RRR}{\mathcal{R}}
\newcommand{\LLL}{\mathcal{L}}

\newcommand{\ars}{\mathcal{R}^*}

\newcommand{\els}{\mathcal{L}^*}

\newcommand{\ar}{{\mathcal{R}}}
\newcommand{\el}{{\mathcal{L}}}

\newcommand{\leqel}{{\leq_{\mathcal{L}}}}

\newcommand{\leqels}{{\leq_{\mathcal{L}^*}}}

\newcommand{\s}{^{\sharp}}

\title{Semigroups of straight left inverse quotients}
\date{\today}
\thanks{This work forms part of the PhD of the second author. It was completed whilst she was supported by a Heilbronn Fellowship.}


\author[V. Gould]{Victoria Gould}

\address{Department of Mathematics, University of York, Heslington, York, YO10 5DD, UK}
\email{victoria.gould@york.ac.uk}
\author[G. Schneider]{Georgia Schneider}
\address{Department of Mathematics, University of Manchester, Oxford Road, Manchester, M13 9PL, UK}
\email{georgia.schneider@manchester.ac.uk}

\begin{document}

\maketitle

\begin{abstract}
Let $Q$ be an inverse semigroup. A subsemigroup $S$ of $Q$ is a \emph{left I-order} in $Q$
and $Q$ is a \emph{semigroup of left I-quotients} of $S$
if every element in $Q$ can be written as $a^{-1}b$, where $a, b \in S$
and $a^{-1}$ is the inverse of $a$ in the sense of inverse semigroup theory.
If we insist on being able to take $a$ and $b$ to be $\RRR$-related in $Q$
we say that $S$ is \emph{straight} in $Q$ and $Q$
is a \emph{semigroup of straight left I-quotients} of $S$.
We give a set of necessary and sufficient conditions for a semigroup to be a
straight left I-order.
The conditions are in terms of two binary relations,
corresponding to the potential restrictions of $\ar$ and $\el$
from an oversemigroup, and an associated partial order. 
Our approach relies on the meet
structure of the \mbox{$\LLL$-classes} of inverse semigroups.
We prove that every finite left I-order is straight
and  give an example of a left I-order which is not straight.
\end{abstract}

\section{Introduction}\label{sec:intro}

Several definitions of a semigroup of quotients
have been proposed and studied by a number of authors, each with their own purpose. 
The earliest definition
is that of a group of left quotients,
introduced by Dubreil in 1943 \cite{dubreil1943problemes},
building on the work of Ore on rings of left quotients \cite{ore1931linear}. 
A subsemigroup $S$ of a group $G$
is a \textit{left order} in $G$ and
$G$ is a \textit{group of left quotients} of $S$
if every $g \in G$ can be written as $g = a^{-1}b$ for some $a,b \in S$.
Ore and Dubreil showed that a semigroup $S$ has a group of left quotients
if and only if $S$ is cancellative and
satisfies the left Ore condition.
The left Ore condition, also known as {\em right reversibility},
says that for any $a,b\in S$ there exists $c,d\in S$ such that $ca=db$.

In 1950 Murata \cite{murata1950quotient}  extended the notions of groups and rings of left quotients to semigroups. He did this 
  by insisting that  the semigroup of quotients
be a monoid, and then by considering inverses lying in the group of units. To be precise,
a subsemigroup $S$ of a monoid $M$
is a \emph{classical\footnote{The terminology {\em classical} is ours.} left order} in $M$ and
$M$ is a \emph{semigroup of classical left quotients} of $S$
if every $m \in M$ can be written as $m = a^{-1}b$ for some $a,b \in S$,
where $a^{-1}$ is the inverse of $a$ in the group of units of $M$,
and if, in addition, every cancellative element of $S$ lies in the group of units of $M$.
Murata showed that a semigroup $S$ has a monoid of classical left quotients if and only if
$S$ satisfies the left Ore-Asano condition, which is the relevant weakening of the left Ore condition. The condition that cancellative elements of $S$ must acquire an inverse in the group of units of $M$ exactly reflects the corresponding condition in classical ring theory.   

Unlike the case for rings and groups, the identity of a monoid may have little influence on its structure. In part to reflect this, a different definition of semigroup of quotients,  and one which does not insist on the oversemigroup being a monoid, was proposed by Fountain and Petrich in 1986 \cite{fountain1986completely}. It was initially 
restricted to completely 0-simple semigroups of (two-sided) quotients; only in the degenerate case are these monoids. 
The first author  \cite{gould1986bisimple} formally extended this concept to
left orders in an arbitrary semigroup, as follows.
A subsemigroup $S$ of a semigroup $Q$
is a \textit{left Fountain-Gould order} in $Q$ and
$Q$ is a \textit{semigroup of left Fountain-Gould quotients} of $S$
if every $q \in Q$ can be written as $q = a^{\s}b$ for some $a,b \in S$,
where $a^{\s}$ is the inverse of $a$ in some subgroup of $Q$,
and if, in addition, every square-cancellable element of $S$ lies in a subgroup of $Q$.
The notion of square-cancellability is a (strong) necessary condition for an element to lie in a subgroup of an oversemigroup, corresponding to the way in which being cancellable is a (strong) necessary condition to lie in the group of units of an oversemigroup. The characterisations of Fountain-Gould left orders in $Q$  become more complicated than in earlier cases, reflecting the $\ar$- and $\el$-class structure of $Q$
\cite{gould2003semigroups}. 

All of the above approaches use, in one form or another, the notion of inverses in a (sub)group. But, in the theory of inverse semigroups, there is another crucial notion of inverse of an element $a$, which we write as $a^{-1}$. The concept central to this paper makes use of this. Namely,  it is that of a semigroup of left inverse quotients,
which we will shorten to left I-quotients,
first defined by Ghroda and Gould in 2010 \cite{ghroda2010inverse}.

\begin{definition}\label{defn:inverseiquotient}
Let $S$ be a subsemigroup of an inverse semigroup $Q$.
Then $Q$ is a \textit{semigroup of left I-quotients} of $S$
and $S$ is a \textit{left I-order} in $Q$
if every $q \in Q$ can be written as
\begin{equation*}
q = a^{-1}b
\end{equation*}
for some $a,b \in S$.
If in addition we can choose $a,b$ such that $a\,\ar\, b$ in $S$, then we say $Q$ is a \textit{semigroup of straight left I-quotients} of $S$
and $S$ is a straight \textit{left I-order} in $Q$.
\end{definition}

The notion of semigroups of I-quotients
has effectively been used by a number of authors
without employing the above terminology.
The first case of this is probably
Clifford in 1953 \cite{clifford1953class}
where he showed that every right cancellative monoid $S$
with the (LC) condition has a bisimple inverse monoid of left I-quotients.
By saying that a semigroup $S$ satisfies the (LC) condition we mean
for any $a, b \in S$ there exists $c \in S$ such that
$Sa \, \cap \, Sb = Sc$.
Thus, (LC) is a stronger condition than right reversibility.
Left I-quotients have also appeared implicitly in work on inverse hulls
of right cancellative semigroups developed in
\cite{nivat1970generalisation} and
\cite{mcalister1976one},
and taken further in \cite{cherubini1987inverse}.
A related approach recently appeared in Exel and Steinberg's work on inverse hulls of
0-left cancellative semigroups \cite{exel2018representations}.
All of these examples are left ample (or right ample),
and so we can determine the structure of their inverse hulls using
Theorem 3.7 of \cite{ghroda2012semigroups}.
Fountain and Kambites effectively utilise left I-quotients in
Section 2 of \cite{fountain2009graph},
in which they use the fact that  graph products of right cancellative, right reversible, monoids are again right cancellative and right reversible, and then  show they are left I-orders in the inverse hull.

None of the earlier work on left I-orders attempts to describe those semigroups that occur as left I-orders. It transpires that, as for left Fountain-Gould orders, this question is more amenable in the straight case. 
The aim of this article, achieved in Theorem~\ref{main}, is to characterise completely those semigroups that occur as straight left I-orders. Again, we will need conditions reflecting the $\ar$- and $\el$-class configurations of the oversemigroup.

The structure of the paper is as follows.
In Section~\ref{sec:observations} we discuss some connections between
left I-orders and left Fountain-Gould orders and
prove that finite left I-orders are always straight.
We also provide some pertinent examples, namely
an example of a left I-order which is not straight
and an example of a semigroup with multiple semigroups of left I-quotients.
Section~\ref{sec:themainresult} contains 
Theorem~\ref{main}, in which we determine the conditions under which a semigroup $S$ is a
straight left I-order.
In Section~\ref{sec:ample} we use our general result
to provide characterisations of right ample and (two-sided) ample straight left I-orders that sit inside their semigroups of left I-quotients in a way that preserves their additional unary operation(s). The result for ample semigroups is particularly pleasing, and after the exigencies of Theorem~\ref{main} brings us to a condition reminiscent of that of Ore.

We assume familiarity with the basic notions of semigroup theory, 
as found
in \cite{howie1995fundamentals} and \cite{clifford1967algebraic}.
In particular, we make frequent use of  Green's relations and their associated preorders.
To avoid ambiguity, we may use the superscript $Q$ to denote the semigroup in question, 
for example, $a \leq_{\LLL^Q} b$ if and only if $Q^1a \subseteq Q^1b$.

\section{Observations on Left I-orders: straightness and uniqueness}\label{sec:observations}

We will make some brief observations concerning left orders
(of various kinds) in inverse semigroups.
Before we begin we remark that
\cite{schein1996subsemigroups} gives an infinite list of
necessary and sufficient conditions for embeddability of a semigroup $S$ into an inverse semigroup.
These may be regarded as analogous to those of Mal'cev
\cite{malcev1939embedding} for embeddability of a semigroup into a group.
Clearly any left I-order must satisfy the conditions of \cite{schein1996subsemigroups}.
However, our focus here is a particular kind of embedding,
and we will produce a finite list of conditions to check;
admittedly, this list is longer than merely cancellativity and right reversibility,
which is all that is required for a semigroup to have a group of left quotients.

Straightness is not only a very useful property,
but one that appears in some typical examples of  left I-orders. 
We will find in Section \ref{sec:themainresult} that if $S$ is straight, we can determine
equalities and products in $Q$
using equalities and relations between elements of $S$.
This makes straight left I-orders easier to work with than general left I-orders.
Because of this, it is of interest to determine when a left I-order
is straight. The following result is an important tool.

\begin{lemma}\label{IntersectLclass}
Let $S$ be a left I-order in $Q$. Then $S$ is straight in $Q$ if and only if
$S$ intersects every $\LLL$-class of $Q$.
\end{lemma}

\begin{proof}
Let $S$ be straight in $Q$, and let $q = a^{-1}b \in Q$ such that $a,b \in S$ and $a \, \RRR^Q \, b$.
Then
\begin{equation*}
q^{-1}q = b^{-1}aa^{-1}b = b^{-1}bb^{-1}b = b^{-1}b,    
\end{equation*}
and so $b \in S \cap L_q$.

Conversely, suppose $S \cap L_q \neq \, \emptyset$ for all $q \in Q$.
Let $q \in Q$; we know that $q = a^{-1}b$, where $a,b \in S$.
Then
\begin{equation*}
q = a^{-1}aa^{-1}bb^{-1}b = a^{-1}fb,
\end{equation*}
where $f = aa^{-1}bb^{-1} \in E(Q)$.
Since $S$ intersects every $\LLL$-class,
there exists ${u \in S \cap L_f}$, and so $f = u^{-1}u$. Hence
\begin{equation*}
(ua)(ua)^{-1} = uaa^{-1}u^{-1} = ufaa^{-1}u^{-1} = ufu^{-1} = uu^{-1}.
\end{equation*}
Similarly $(ub)(ub)^{-1} = uu^{-1}$.
We can therefore write
\begin{equation*}
q = a^{-1}fb =a^{-1}u^{-1}ub = (ua)^{-1}(ub), 
\end{equation*}
where $ua \, \RRR^Q\, u\, \RRR^Q \, ub$.
It follows that $Q$ is straight over $S$.
\end{proof}

We now briefly compare the different kinds of left orders in inverse semigroups. Observe that if an inverse monoid has trivial group of units, then it contains no proper classical left order. However, if $S$ is a Fountain-Gould left order in an inverse semigroup $Q$
(and many are known to exist, see, for example \cite{fountain1985brandt}),
then $S^1$ is easily seen to be  a Fountain-Gould left order in $Q^1$. On the other hand,
we have the following.

\begin{lemma}\label{fg-i}
Let $S$ be a left Fountain-Gould order  an inverse semigroup $Q$.
Then $S$ is a straight left I-order in $Q$.
\end{lemma}

\begin{proof}
Let $a \in S$. If $a^{\#}$ exists, then clearly $a^{-1} = a^{\#}$.
We have that every $q \in Q$ can be written as $q = a^{\#}b$ for some $a, b \in S$.
Therefore every $q \in Q$ can be written as $q = a^{-1}b$ for some $a, b \in S$.

We will now prove that $S$ is straight in $Q$
by proving that $S$ intersects every $\LLL$-class of $Q$.
Let $q \in Q$.
We know that $q \in Q$ can be written as $q = a^{\#}b$ for some $a, b \in S$.
We see that
\begin{equation*}
q = a^{\#}b \, \LLL^Q \, aa^{\#}b = a^{\#}ab \, \LLL^Q \, ab.
\end{equation*}
We know that $ab \in S$ and so  $ab \in L_q \cap S$.
Therefore, $S$ intersects every $\LLL$-class of $Q$.
We apply Lemma \ref{IntersectLclass} to obtain that
$S$ is straight in $Q$.
\end{proof}

The converse to Lemma~\ref{fg-i} is not true. 

\begin{example}
Let $\mathcal{B}$ be the bicyclic monoid, which for convenience we write as elements of $\mathbb{N}^0\times \mathbb{N}^0$. Let $S$ be the
$\mathcal{R}$-class of the identity,
${S = \{ \, (0,n) \, | \, n \in \mathbb{N}^0 \, \} }$.
Certainly $\mathcal{B}$ is an inverse semigroup and for any $(a,b) \in \mathcal{B}$,
\begin{equation*}
(a,b) = (a,0)(0,b) = (0,a)^{-1}(0,b)
\end{equation*}
so $\mathcal{B}$ is a semigroup of left I-quotients of $S$.
Since $S$ intersects every $\LLL$-class of $\mathcal{B}$ we have from Lemma \ref{IntersectLclass} that $S$ is straight. 

The monoid $\mathcal{B}$ has trivial subgroups,
and as the only element of $S$ that lies in a subgroup is $(0,0)$,
it is clear that $S$ is not a Fountain-Gould left order in $\mathcal{B}$.
\end{example}

We now show that for a wide class of inverse semigroups, all left I-orders are straight. 

\begin{theorem}\label{thm:css} Let $Q$ be a completely semisimple semigroup with the ascending chain condition on the $\leq_{\mathcal{J}}$-order. Then any left I-order in $Q$ is straight.
\end{theorem}
\begin{proof}  Let $S$ be a left I-order in $Q$, and suppose that  $S$ is not straight,
so that by Lemma~\ref{IntersectLclass},
$S$ does not intersect every $\LLL$-class of $Q$.
Let $J$ be a $\mathcal{J}^Q$-class  maximal with respect to 
there being an $\LLL^Q$-class $L$ in $ J$ such that $S\cap L=\emptyset$. Let $L=L^Q_e$ where $e=e^2$.  Since $S$ is a left I-order in $Q$, we know that there exists $a,b \in S$
such that
\begin{equation*}
    e = a^{-1}b=b^{-1}a,
\end{equation*}
the second statement following from the fact that 
 $e$ is an idempotent. Now $e\leq_{\LLL^Q} a$ and as 
 $e\neq a^{-1}a$ we have that $e<_{\LLL^Q} a^{-1}a$ and then by the semisimplicity of $Q$ that  $e<_{\mathcal{J}^Q} a^{-1}a$. By assumption, $S$ intersects every $\LLL^Q$-class of $J^Q_{a^{-1}a}=J^Q_a$. In particular, there is a $c\in S$ such that $c^{-1}c=aa^{-1}$ and then
 \[  e=a^{-1}b\,\LLL^Q\, aa^{-1}b=c^{-1}cb\, \LLL^Q\, cb.\]
 But $cb\in S$, a contradiction.

\end{proof} 

\begin{corollary}\label{cor:straight}  Let $Q$ be a finite inverse semigroup.  Then any left I-order in $Q$ is straight.\end{corollary}

An obvious question to ask is whether every left I-order is straight.
The answer to this is no, although they do seem to be difficult to find.
The authors are grateful to Mark Kambites for the following  example. 

\begin{example}\label{nonstraight}
Let $X$ be a countably infinite set
and let $Q = \mathcal{I}_X$ be the symmetric inverse monoid on $X$.
Let $S$ be the set of surjective elements of $Q$.
Then $S$ is a left I-order in $Q$ which is not straight.
\end{example}

\begin{proof}
We start by proving that $S$ is a subsemigroup of $Q$.
Let $a,b \in S$ so that $\text{im}(a) = \text{im}(b) = X$.
We have that
\begin{equation*}
\text{im}(ab) = (\text{im}(a) \cap \text{dom}(b))b
= (X \cap \text{dom}(b))b = (\text{dom}(b))b = \text{im}(b) = X.
\end{equation*}
Therefore $ab \in S$ and so $S$ is a subsemigroup of $Q$.

We now prove that $S$ in a left I-order in $Q$.
Let $\gamma \in Q$ with $\text{dom}(\gamma) = A$ and $\text{im}(\gamma) = B$.
Let $C = X \backslash A$ and let $D = X \backslash B$.
Our aim is to find $\alpha, \beta \in S$ such that $\gamma = \alpha^{-1} \beta$.

\textbf{Case 1}: $C$ and $D$ both infinite.

Split $C$ into two disjoint infinite sets, $C = V \, \dot{\cup} \, W$.
We know that $C,D,V,W$ are all countably infinite, so we can construct
two bijections,
$\lambda: W \rightarrow C$ and
$\mu: V \rightarrow D$.

We then define $\alpha: A \, \dot{\cup} \, W \rightarrow X$ by
\begin{equation*}
x \alpha =
\begin{cases}
x &\text{if $x \in A$}\\
x\lambda &\text{if $x \in W$},
\end{cases}
\end{equation*}
and $\beta: A \, \dot{\cup} \, V \rightarrow X$ by
\begin{equation*}
x \beta =
\begin{cases}
x\gamma &\text{if $x \in A$}\\
x\mu &\text{if $x \in V$}.
\end{cases}
\end{equation*}
It is easy to check that $\alpha$ and $\beta$ are surjective and one-to-one.
We will now prove that $\gamma = \alpha^{-1} \beta$.
We see that
\begin{equation*}
\begin{split}
\text{dom}(\alpha^{-1}\beta)
&= (\text{im}(\alpha^{-1}) \cap \text{dom}(\beta))\alpha\\
&= (\text{dom}(\alpha) \cap \text{dom}(\beta))\alpha\\
&= (( A \, \dot{\cup} \, W) \cap ( A \, \dot{\cup} \, V))\alpha\\
&= A\alpha = A = \text{dom}(\gamma),
\end{split}
\end{equation*}
and that for $x \in A$,
we have $x \alpha^{-1}\beta = x \beta = x \gamma$.
Therefore $\gamma = \alpha^{-1} \beta$.

\textbf{Case 2}: $C$ infinite and $D$ finite.

Split $C$ into two disjoint sets, $C = U \, \dot{\cup} \, W$,
with $|U| = |D|$ and $W$ infinite.
We know that $C$ and $W$ are countably infinite, so we can construct
a bijection $\lambda: W \rightarrow C$.
Also since $|U| = |D|$, we can construct a bijection
$\mu: U \rightarrow D$.

We define $\alpha: A \, \dot{\cup} \, W \rightarrow X$ by
\begin{equation*}
x \alpha =
\begin{cases}
x &\text{if $x \in A$}\\
x\lambda &\text{if $x \in W$},
\end{cases}
\end{equation*}
and $\beta: A \, \dot{\cup} \, U \rightarrow X$ by
\begin{equation*}
x \beta =
\begin{cases}
x\gamma &\text{if $x \in A$}\\
x\mu &\text{if $x \in U$}.
\end{cases}
\end{equation*}
In a similar way to the previous case,
$\alpha$ and $\beta$ are surjective with $\gamma = \alpha^{-1} \beta$.

\textbf{Case 3}: $C$ finite and $D$ infinite.

By considering $\gamma^{-1}:B \rightarrow A$, we can use Case 2 to write
$\gamma^{-1} = \alpha^{-1}\beta$ with $\alpha, \beta$ surjective.
Then $\gamma = \beta^{-1}\alpha$.

\textbf{Case 4}: $C$ and $D$ both finite.

Since $C$ is finite, we know that $A = X \backslash C$ is countably infinite.
Therefore we can construct
a bijection $\beta: A \rightarrow X$.
We see that
\begin{equation*}
\text{dom}(\beta\gamma^{-1})
= (\text{im}(\beta) \cap \text{dom}(\gamma^{-1}))\beta^{-1}
= (\text{im}(\beta) \cap \text{im}(\gamma))\beta^{-1}
= B \beta^{-1}.
\end{equation*}
Note that $B \beta^{-1} \subseteq A$, and so $B \beta^{-1}$ is disjoint with $C$.
We define $\alpha: C \, \dot{\cup} \, B \beta^{-1} \rightarrow X$ by
\begin{equation*}
x \alpha =
\begin{cases}
x &\text{if $x \in C$}\\
x(\beta\gamma^{-1}) &\text{if $x \in B \beta^{-1}$}.
\end{cases}
\end{equation*}
We firstly prove that $\alpha$ is surjective.
We see that $\alpha$ is a piecewise composition of two functions.
The first piece has image $C$. We calculate the image of the second piece as follows.
\begin{equation*}
\text{im}(\beta\gamma^{-1})
= (\text{im}(\beta) \cap \text{dom}(\gamma^{-1}))\gamma^{-1}
= (\text{im}(\beta) \cap \text{im}(\gamma))\gamma^{-1}
= B \gamma^{-1} = A.
\end{equation*}
Therefore $\text{im}(\alpha) = A \, \dot{\cup} \, C = X$.
Note that since the domains and images
of the two pieces of $\alpha$ are disjoint
and both pieces are one-to-one, it follows that $\alpha$ is one-to-one.

We will now prove that $\gamma = \alpha^{-1} \beta$.
We see that
\begin{equation*}
\begin{split}
\text{dom}(\alpha^{-1}\beta)
&= (\text{im}(\alpha^{-1}) \cap \text{dom}(\beta))\alpha\\
&= (\text{dom}(\alpha) \cap \text{dom}(\beta))\alpha\\
&= ((C \, \dot{\cup} \, B \beta^{-1}) \cap A)\alpha\\
&= (B \beta^{-1})\alpha = \text{im}(\beta\gamma^{-1}) = A = \text{dom}(\gamma),
\end{split}
\end{equation*}
and that for $x \in A$ we have
 \[x \alpha^{-1}\beta = x(\beta \gamma^{-1})^{-1}\beta=x \gamma \beta^{-1} \beta = x \gamma.\]
Therefore $\gamma = \alpha^{-1} \beta$.

We have now proved that $S$ is a left I-order in $Q$.
We know that $S$ is not straight in $Q$
because it does not intersect every $\LLL$-class of $Q$.
Indeed, although $Q$ has an infinite number of $\LLL$-classes,
one for each subset of $X$,
$S$ is contained completely within the largest one.
\end{proof}

Another natural question to ask is whether a given semigroup has at most one semigroup of left I-quotients. 
The answer to this is no, as we show in the following example.

\begin{example}
Section 4 of \cite{gould1986clifford}
provides an example of a semigroup with two non-isomorphic Clifford semigroups
of left Fountain-Gould quotients.
For $Q$, a Clifford semigroup,
we know that every element lies in a subgroup.
Consequently, for every $a \in Q$, $a^{\#}$ exists and $a^{-1} = a^{\#}$.
We can therefore see that a semigroup $S$ is a left I-order in $Q$ if and only if
$S$ is a left Fountain-Gould order in $Q$.
\end{example}

\section{The general case}\label{sec:themainresult}

The aim of this section is to prove Theorem \ref{main},
which gives necessary and sufficient conditions for
a semigroup $S$ to be a straight left I-order.
Our first lemma follows immediately from the fact that in an inverse semigroup $Q$ we have $Qe\cap Qf=Qef$ for any $e,f\in E(Q)$  so that 
 the poset $Q / \LLL$ is order isomorphic to
the semilattice of idempotents under the natural partial order.

\begin{lemma}\label{Invmeet}
Let $Q$ be an inverse semigroup.
Then $Q / \LLL$ is a meet semilattice with
$L_a \, \wedge \, L_b = L_c$ if and only if $c^{-1}c = a^{-1}a b^{-1}b$.
\end{lemma}

Assume now that $S$ has a semigroup of straight left I-quotients $Q$.
We aim to identify properties of $S$ inherited from $Q$ with the eventual
goal of reconstructing $Q$ from these properties.

By definition, every element in $Q$ can be written as $a^{-1}b$, where $a,b \in S$
and $a \, \RRR^Q \, b$.
Therefore, we can reconstruct $Q$ as ordered pairs of elements of $S$
under an equivalence relation:
\begin{equation*}
Q \cong \{ (a,b) \, | \, a,b \in S, \, a \, \RRR^Q \, b \} / \sim \,
\end{equation*}
where \[(a,b)\sim(c,d)\mbox{ if and only if }a^{-1}b = c^{-1}d\mbox{ in }Q.\]
This relation has already been determined by Ghroda and Gould.
\begin{lemma}{\cite[Lemma 2.7]{ghroda2012semigroups}}\label{GhrodaGould1}
Let $S$ be a straight left I-order in $Q$. Let ${a,b,c,d \in S}$
with $a \, \RRR^Q \, b$ and $c \, \RRR^Q \, d$.
Then $a^{-1}b = c^{-1}d$ in $Q$ if and only if there exists $x,y \in S$ such that
\begin{equation*}
xa=yc, \; xb=yd, \; x \, \RRR^Q \, y, \; x^{-1} \, \RRR^Q \, a
\text{ and } \, y^{-1} \: \RRR^Q \, c.
\end{equation*}
\end{lemma}
However, we wish to be able to express the conditions in Lemma \ref{GhrodaGould1}
entirely in terms of elements of $S$.
We remind the reader that in an inverse semigroup $Q$,
we have that
$x \, \RRR^Q \, y$ if and only if $x^{-1} \, \LLL^Q \, y^{-1}$.

\begin{lemma}\label{box}
Let $Q$ be an inverse semigroup and let ${x,a \in Q}$.
Then
\begin{equation*}
x^{-1} \, \RRR^Q \, a \, \text{ if and only if } \, x \, \RRR^Q \, xa \, \LLL^Q \, a.
\end{equation*}
\end{lemma}
\begin{proof}
Let $x^{-1} \, \RRR^Q \, a$.
Using the fact that $\RRR^Q$ is a left congruence,
this implies
\begin{equation*}
xa \, \RRR^Q \, xx^{-1} \, \RRR^Q \, x.  
\end{equation*}
We know that $x^{-1} \, \RRR^Q \, a$ implies that $x \, \LLL^Q \, a^{-1}$.
Therefore, using the fact that $\LLL^Q$ is a right congruence,
we also have
\begin{equation*}
xa \, \LLL^Q \, a^{-1}a \, \LLL^Q \, a.  
\end{equation*}

Conversely, let $xa \in R_x \cap L_a$. By \cite[Prop. 2.3.7]{howie1995fundamentals},
we have that $L_x \cap R_a$ contains an idempotent, $e$.
\[
\begin{array}{|l|lll|l|}
\hline
xa &  &  &  & x \\ \hline
   &  &  &  &   \\
   &  &  &  &   \\ \hline
a  &  &  &  & e\\ \hline
\end{array}\] 

Then, as $x \, \LLL^Q \, e$, we have
$x^{-1} \, \RRR^Q \, e^{-1} = e \, \RRR^Q \, a$.
\end{proof}
We now rewrite Lemma \ref{GhrodaGould1} in terms of relations restricted to $S$.

\begin{lemma}\label{ERdefn}
Let $S$ be a straight left I-order in $Q$. Let $a,b,c,d \in S$
with $a \, \RRR^Q \, b$ and $c \, \RRR^Q \, d$.
Then $a^{-1}b = c^{-1}d$ if and only if there exists $x,y \in S$ such that
\begin{equation*}
xa=yc, \; xb=yd, \; x \, \RRR^Q \, xa \, \LLL^Q \, a,
\text{ and } \: y \, \RRR^Q \, yc \, \LLL^Q \, c.
\end{equation*}
Note that since $xa=yc$, the conditions imply that
$x \, \RRR^Q \, y$ and $a \, \LLL^Q \, c$.
\end{lemma}
The conditions given in Lemma \ref{ERdefn} will determine our $\sim$.

The next thing to address is multiplication on $Q$.
We note that for every $b, c \in S$, $bc^{-1} \in Q$
and therefore, since $Q$ is a semigroup of straight left I-quotients of $S$,
there exists $u,v \in S$ with $u \, \RRR^Q \, v$, such that
$bc^{-1} = u^{-1}v$ in $Q$.
Therefore, multiplication on $Q$ {\em must be } given by
$a^{-1}bc^{-1}d = (ua)^{-1}(vd)$, where
$bc^{-1} = u^{-1}v$ in $Q$.
In the same way as we need to internalise to $S$ the condition that $a^{-1}b=c^{-1}d$ in $Q$,
we need to be able to express the equality $bc^{-1}=u^{-1}v$ solely in terms of elements of $S$. We first quote a result from 
\cite{ghroda2012semigroups}.

\begin{lemma}\cite[Lemma 2.6]{ghroda2012semigroups}\label{GhrodaGould2}
Let $b, c, u, v$ be elements of an inverse semigroup Q such that $u \, \RRR^Q \, v$.
If $bc^{-1} = u^{-1}v$ then $ub = vc$.
\end{lemma}

\begin{lemma}\label{ORE}
Let $Q$ be an inverse semigroup and let $b,c,u,v \in Q$ such that $u \, \RRR^Q \, v$.
Then $bc^{-1} = u^{-1}v$ in $Q$ if and only if
\begin{equation*}
ub = vc, \; v \, \RRR^Q \, vc \, \text{ and }
\, L^Q_b \, \wedge \, L^Q_c = L^Q_{ub}.  
\end{equation*}
\end{lemma}

\begin{proof}
Let $bc^{-1} = u^{-1}v$.
By Lemma \ref{GhrodaGould2}, we have $ub = vc$.
Since $u \, \RRR^Q \, v$, we know that $uu^{-1} = vv^{-1}$. Therefore,
using $u^{-1}v = bc^{-1}$ and $ub = vc$, we have
\begin{equation*}
v = vv^{-1}v = uu^{-1}v = ubc^{-1} = vcc^{-1}
\end{equation*}
and therefore $v \, \RRR^Q \, vc$.
Finally, again using $bc^{-1} = u^{-1}v$ and $ub = vc$,
we have
\begin{equation*}
b^{-1}bc^{-1}c = b^{-1}u^{-1}vc = b^{-1}u^{-1}ub = (ub)^{-1}(ub). 
\end{equation*}
It follows from  Lemma \ref{Invmeet}  that
$L^Q_b \, \wedge \, L^Q_c = L^Q_{ub}$.

Conversely, let
\begin{equation*}
ub = vc, \; v \, \RRR^Q \, vc \, \text{ and }
\, L^Q_b \, \wedge \, L^Q_c = L^Q_{ub}.  
\end{equation*}
Since $Q$ is an inverse semigroup $v \, \RRR^Q \, vc$
implies that $vv^{-1} = vcc^{-1}v^{-1}$,
and so $v = vcc^{-1}$.
Using this along with $vc = ub$, we have
\begin{equation}\label{uv}
u^{-1}v = u^{-1}vcc^{-1} = u^{-1}ubc^{-1}
= u^{-1}ubb^{-1}bc^{-1}
=bb^{-1}u^{-1}ubc^{-1},
\end{equation}
using the fact that idempotents commute
in the last equality.
By Lemma \ref{Invmeet}, we know that
$\, L^Q_b \, \wedge \, L^Q_c = L^Q_{ub} \,$ implies that
$\, {b^{-1}bc^{-1}c = (ub)^{-1}(ub)}$.
Therefore
\begin{equation}\label{bc}
bb^{-1}u^{-1}ubc^{-1} = b(ub)^{-1}(ub)c^{-1}
= bb^{-1}bc^{-1}cc^{-1} = bc^{-1}.
\end{equation}
Putting Equations (\ref{uv}) and (\ref{bc}) together,
we obtain $u^{-1}v = bc^{-1}$.
\end{proof}

We now introduce the notation used
in Theorem \ref{main}. 
The relation $\RRR^*$ is defined on a semigroup $S$
by the rule that $a \, \RRR^* \, b$
if and only $a \, \RRR \, b$ in some oversemigroup of $S$.
According to Lemma 1.7 of \cite{mcalister1976one},
$a \, \RRR^* \, b$ is equivalent to the condition that
$xa = ya$ if and only if $xb = yb$ for all $x, y \in S^1$.
By considering the right regular embedding of $S$ as a subsemigroup of $\mathcal{T}_{S^1}$, it is easy to see that $a \, \RRR^* \, b$ if and only if $a\,\RRR\, b$ in $\mathcal{T}_{S^1}$.
The relation $\RRR^*$ will always refer to $S$. For a given preorder $\, \leq_l$ on $S$
we will use $\LLL'$ to denote the associated equivalence relation
and  $L'_a$ to denote the $\LLL'$-class of $a$. We use 
 $\wedge$ to denote the meet of $\LLL'$-classes
with respect to the partial order induced by  $\leq_l$, where these meets exist.
We state our theorem making use of  Greek letters, which will aid in applicability and reference,  when we will use Roman letters for fixed elements of our semigroup.

\begin{theorem}\label{main}
Let $S$ be a semigroup and let $\RRR'$ and $\, \leq_l$ be binary relations on $S$.
Then $S$ has a semigroup of straight left I-quotients
$Q$ such that \[{\RRR^Q \cap (S \times S) = \RRR'}\mbox{ and }
\leq_{\LLL^Q} \cap \, (S \times S) = \, \leq_l\]
if and only if $\RRR'$ is a left compatible equivalence relation;
$\, \leq_l$ is a preorder
such that the $\LLL'$-classes form a meet semilattice under the associated partial order;
and $S$ satisfies Conditions \mbox{(M1) - (M6)}.

\begin{itemize}
\item[(M1)] For all $\alpha, \beta \in S$,
there exists $\gamma, \delta \in S$ such that
\begin{equation*}
\gamma \, \RRR' \, \delta \, \RRR' \, \delta \beta = \gamma \alpha \, \text{ and } \,
L'_\alpha \, \wedge \, L'_\beta = L'_{\gamma \alpha}.    
\end{equation*}
\item[(M2)] Right multiplication distributes over meet, that is,
for all $\alpha, \beta, \gamma, \delta \in S$,
\begin{equation*}
{L'_\alpha \, \wedge \, L'_\beta = L'_\gamma} \, \text{ implies that } \,
L'_{\alpha\delta} \, \wedge \, L'_{\beta\delta} = L'_{\gamma\delta}.
\end{equation*}
\item[(M3)] For all $\alpha, \beta \in S$,
$\alpha \beta \leq_l \beta$.
\item[(M4)] $\RRR' \subseteq \RRR^*$.
\item[(M5)] Let $\alpha, \beta, \gamma, \delta \in S$ such that
$\gamma \, \RRR' \, \gamma \alpha \, \LLL' \, \alpha$ and
$\delta \, \RRR' \, \delta \beta \, \LLL' \, \beta$.
Then $\gamma \, \LLL' \, \delta$ if and only if $\alpha \, \RRR' \, \beta$.
\item[(M6)] For all $\alpha, \beta, \gamma \in S$,
$\alpha \, \LLL' \, \beta \, \LLL' \, \gamma \alpha = \gamma \beta$
implies that $\alpha = \beta$.
\end{itemize}
\end{theorem}
We start the proof of Theorem \ref{main} by first proving the forward implication.
We assume that $S$ has a semigroup of straight left I-quotients, $Q$,
and we put ${\RRR^Q \cap (S \times S) = \RRR'}$,
$\, {\LLL^Q \cap (S \times S) = \LLL'}$ and
$\, \leq_{\LLL^Q} \, \cap \, (S \times S) = \, \leq_l \,$.
From knowledge of Green's relations,
we know that $\RRR'$ is a left congruence on $S$,
and that
$\leq_l$ is a preorder on $S$
with the associated equivalence
relation, $\LLL'$.
Using Lemma \ref{Invmeet}, we know that
$Q / \LLL^Q$ forms a meet semilattice under $\leq_{\LLL^Q}$.
Since $S$ intersects every $\LLL^Q$-class, this means that
$S / \LLL'$ forms a meet semilattice under $\leq_{l}$.
We now prove that Properties (M1) - (M6) hold.

~\begin{itemize}
\item[(M1)] Let $\alpha,\beta \in S$. Then $\alpha,\beta \in Q$ and so,
by closure under taking of inverses and multiplication, $\alpha \beta^{-1} \in Q$.
Since $Q$ is a semigroup of straight left \mbox{I-quotients} of $S$, there exists
$\gamma, \delta \in S$ such that $\alpha \beta^{-1} = \gamma^{-1}\delta$ with
$\gamma \, \RRR' \, \delta$.
Lemma \ref{ORE} then gives the result.
\item[(M2)] Since $Q$ is an inverse semigroup, we can use
Lemma \ref{Invmeet} to give us that
$\, L'_\alpha \, \wedge \, L'_\beta = L'_\gamma \,$ is equivalent to
$\, \alpha^{-1}\alpha \beta^{-1}\beta = \gamma^{-1}\gamma$.
Therefore
\begin{equation*}
\begin{split}
(\alpha\delta)^{-1}(\alpha\delta)(\beta\delta)^{-1}(\beta\delta)
& = \delta^{-1} \alpha^{-1}\alpha \delta \delta^{-1} \beta^{-1}\beta \delta\\
& = \delta^{-1} \alpha^{-1}\alpha \beta^{-1}\beta \delta\\
& = \delta^{-1} \gamma^{-1}\gamma \delta\\
& = (\gamma\delta)^{-1}(\gamma\delta).
\end{split}
\end{equation*}
And so, using Lemma \ref{Invmeet} again,
$L'_{\alpha\delta} \, \wedge \, L'_{\beta\delta} = L'_{\gamma\delta}$.
\item[(M3)] This is true in any semigroup, since $Q^1 \alpha \beta \subseteq Q^1 \beta$.
\item[(M4)] Since $\RRR' = \RRR^Q \cap (S \times S)$
and $Q$ is an oversemigroup of $S$,
then, by definition, 
$\alpha \, \RRR' \, \beta$ implies that $\alpha \, \RRR^* \, \beta$.
\item[(M5)] By Lemma \ref{box}, in an inverse semigroup we have that
$\gamma \, \RRR' \,  \gamma \alpha \, \LLL' \,  \alpha$
implies that $\gamma^{-1} \, \RRR^Q \, \alpha$, and similarly
$\delta \, \RRR' \,  \delta \beta \, \LLL' \, \beta$
implies that $\delta^{-1} \, \RRR^Q \, \beta$.
Then $\alpha \, \RRR' \, \beta$ implies that
$\gamma^{-1} \, \RRR^Q \, \alpha \, \RRR^Q \, \beta \, \RRR^Q \, \delta^{-1}$.
We know that $\gamma^{-1} \, \RRR^Q \, \delta^{-1}$ implies that
$\gamma \, \LLL^Q \, \delta$, and so $\gamma \, \LLL' \, \delta$.
The converse is similar.
\item[(M6)] Since $\alpha$ and $\gamma$ are elements in an inverse semigroup,
$\alpha \, \LLL' \, \gamma \alpha$ if and only if
${\alpha= \gamma^{-1}\gamma\alpha}$.
Similarly, $\beta \, \LLL' \, \gamma\beta$
if and only if ${\beta = \gamma^{-1}\gamma\beta}$.
Therefore, $\gamma\alpha = \gamma\beta \,$
together with $\alpha \, \LLL' \, \gamma \alpha$
and $\beta \, \LLL' \, \gamma\beta$, implies that
\begin{equation*}
\alpha=\gamma^{-1}\gamma\alpha = \gamma^{-1}\gamma\beta = \beta.
\end{equation*}
\end{itemize}

This proves the forward implication of Theorem \ref{main}.
Note that in the above we only required that $S$ is embedded in $Q$ to obtain (M2) - (M6);
for (M1) we required $S$ to be a straight left I-order.

We now prove the converse.
This will consist of proving that the following construction, $P$,
yields a semigroup of straight left I-quotients of $S$,
with ${\RRR' = \RRR^P \cap \, (S \times S)}$ and
$\leq_l \, = \, \leq_{\LLL^P} \cap \, (S \times S)$.
For the convenience of the reader, we now set up the `roadmap'
for the proof.

\begin{roadmap}\label{roadmap}
Let $S$ be a semigroup with $\RRR'$, $\leq_l$ and $\LLL'$
satisfying the conditions of Theorem \ref{main}.
Note that by considering (M2) with $\alpha = \gamma$,
we see that $\leq_l$ is right compatible.
Therefore, since $\LLL'$ is an equivalence relation
associated with a right compatible preorder,
$\LLL'$ is a right congruence.

We begin by defining
\begin{equation*}
\Sigma = \{ (a,b) \in S \times S \, | \, a \, \RRR' \, b \}.
\end{equation*}
We then define an equivalence relation $\sim$ on $\Sigma$, by
\begin{equation*}
(a,b)\sim(c,d)
\end{equation*}
if and only if there exists $x,y \in S$ such that
\begin{equation*}
xa=yc, \, xb=yd, \; x \, \RRR' \, xa \, \LLL' \, a\mbox{ and } y \, \RRR' \, yc \, \LLL' \, c.
\end{equation*}
Note that $x \, \RRR' \, y$ and $a \, \LLL' \, c$ as a consequence.

We show that this is an equivalence relation in Lemma \ref{ER}.
Using $[a,b]$ to denote the equivalence class of  $(a,b)$ under $\sim$ we then define $P = \Sigma \, / \! \sim$ and multiplication on $P$ via the following rule:
\begin{equation*}
{[a,b][c,d] = [ua,vd]}
\text{ where }
{\, u \, \RRR' \, v \, \RRR' \, vc = ub}\mbox{ and } L'_b \, \wedge \, L'_c = L'_{ub} \,.
\end{equation*}

Note that such a $u$ and $v$ exist in $S$ by (M1).
We show that $P$ is a semigroup in Lemma \ref{WD} and Lemma \ref{Ass}
and an inverse semigroup in Lemma \ref{regular} and Lemma \ref{Inv}.

The next step is to show that $S$ embeds into $P$,
by defining $\phi :S \rightarrow P$ by
$a \phi = [x,xa]$, where $x$ is an element in $S$
such that $x \, \RRR' \, xa \, \LLL' \, a$.
The existence of such an $x$ is a consequence of (M1) proved in Lemma \ref{secprops}.
We will prove that this function is an embedding in Lemma \ref{EB}.

We show that
the restriction of $\RRR^P$ to $(S \times S)$ is $\RRR'$
in Lemma \ref{Rproof}, and that
the restriction of $\, \leq_{\LLL^P}$ to $(S \times S)$ is $\, \leq_l$
in Lemma \ref{Lproof}.
Finally, we verify that $P$ is a
semigroup of straight left I-quotients of $S \phi$
in Lemma \ref{ssliq}.
\end{roadmap}

Now that we have set up the `roadmap', the rest of the section
will be the `road trip'.
The properties in Theorem \ref{main} will be used extensively,
so the reader might prefer to have the list of properties in front of them
whilst reading.
For the remaining results in this section,
$S$, $\RRR'$, $\leq_l$ and $\LLL'$
are as described in the conditions of Theorem \ref{main},
and $\Sigma$, $\sim$, $P$ and $\phi$ are as described in
Roadmap \ref{roadmap}.

The following technical lemma will provide a few shortcuts as we procede with the main proof.

\begin{lemma} \label{secprops}
~ \begin{enumerate}[(i)]
\item For all $a \in S$, there exists an $x \in S$ such that
$x \, \RRR' \, xa \, \LLL' \, a$.
\item For all $a,b,x \in S$,
$x \, \RRR' \, xa \, \LLL' \, a$ and $\, a \, \RRR' \, b$ 
together imply that
$x \, \RRR' \, xb \, \LLL' \, b$.
\item For all $x,a \in S$, we have $L'_{xa} \wedge L'_a = L'_{xa}$.
\item For all $a,b,x,y \in S$,
$a \, \RRR' \, b$ and $xa\, \LLL' \, ya$
together imply that $xb \, \LLL' \, yb$.
\end{enumerate}
\end{lemma}

\begin{proof}
~ \begin{enumerate}[(i)]
\item By applying (M1) with $\alpha = \beta = a$,
there exists $x \in S$ such that $x \, \RRR' \, xa$
and $L'_{xa} = L'_a \, \wedge \, L'_a = L'_a$.
\item Let $a,b,x \in S$ such that $\, a \, \RRR' \, b$ and
$x \, \RRR' \, xa \, \LLL' \, a$.
Using the fact that $\RRR'$ is a left congruence,
$b \, \RRR' \, a$ implies that $xb \, \RRR' \, xa \, \RRR' \, x$.
By (i), there exists ${y \in S}$ such that $y \, \RRR' \, yb \, \LLL' \, b$.
We can then use (M5),
to see that $a \, \RRR' \, b$ implies that $x \, \LLL' \, y$.
Therefore, using the fact that $\LLL'$ is a right congruence,
$xb \, \LLL' \, yb \, \LLL' \, b$.
\item Let $x,a \in S$.
By (M3), we know that $\, xa \leq_l a$,
Therefore, by the definition of meet, we have that
$L'_{xa} \wedge L'_a = L'_{xa}$.
\item Applying (M1) to $x a$ and $y a$,
there exists $w,z \in S$ such that
\begin{equation*}
w \, \RRR' \, z \, \RRR' \, zy a = wx a \, \text{ and }
\, L'_{x a} \wedge L'_{y a} = L'_{w x a} .
\end{equation*}
Since $x a \, \LLL' \, y a$,
this gives us
\begin{equation*}
w \, \RRR' \, w x a \, \LLL' \, x a
\text{ and } z \, \RRR' \, zy a \, \LLL' \, y a.
\end{equation*}
Using the fact that $\RRR'$ is a left congruence,
we have that $a \, \RRR' \, b$ implies that both
$x a \, \RRR' \, x b$
and $y a \, \RRR' \, y b$.
Therefore we can apply (ii) to both of the above equations to get
\begin{equation*}
w \, \RRR' \, w x b \, \LLL' \, x b
\text{ and } z \, \RRR' \, zy b \, \LLL' \, y b.
\end{equation*}
Also we can apply (M4) to $wx a = zy a$ to get
$wx b = zy b$.
Therefore $x b \, \LLL' \, w x b
= z y b \, \LLL' \, y b$.

\end{enumerate}
\end{proof}

\begin{lemma}\label{ER}
The relation $\sim$ is an equivalence relation on $\Sigma$.
\end{lemma}

\begin{proof} We check the required properties in turn.

\textbf{Reflexivity}: Let $(a,b) \in \Sigma$. By definition,
$(a,b) \sim (a,b)$ if there exists $x,y \in S$ such that
\[xa=ya, \, xb=yb, \; x \, \RRR' \, xa \, \LLL' \, a\mbox{ and } y \, \RRR' \, ya \, \LLL' \, a.\]
By Lemma \ref{secprops} (i), there is an $x \in S$ such that
$x \, \RRR' \, xa \, \LLL' \, a$. Then take $y = x$ to get
reflexivity.

\textbf{Symmetry}: Let $(a,b) \sim (c,d)$. By definition there must exist $x,y \in S$
such that
\begin{equation*}
xa=yc, \, xb=yd, \; x \, \RRR' \, xa \, \LLL' \, a\mbox{ and } y \, \RRR' \, yc \, \LLL' \, c,
\end{equation*}
By switching the roles of $x$ and $y$, we can immediately see that $(c,d) \sim (a,b)$.

\textbf{Transitivity}: Let $(a,b) \sim (c,d)$. Therefore there exists $x,y \in S$ such that
\begin{equation}\label{xayc}
xa=yc, \, xb=yd, \; x \, \RRR' \, xa \, \LLL' \, a\mbox{ and }  y \, \RRR' \, yc \, \LLL' \, c.
\end{equation}
Suppose also that $(c,d) \sim (e,f)$. Then there exists $w,z \in S$ such that
\begin{equation}\label{wcze}
wc=ze, \, wd=zf, \; w \, \RRR' \, wc \, \LLL' \, c\mbox{ and }  z \, \RRR' \, ze \, \LLL' \, e.
\end{equation}
We need $(a,b) \sim (e,f)$. That is, we need $X,Y \in S$ such that
\begin{equation}\label{XaYe}
X\!a=Y\!e, \, X\!b=Y\!f, \; X \, \RRR' \, X\!a \, \LLL' \, a\mbox{ and }  Y \, \RRR' \, Y\!e \, \LLL' \, e.
\end{equation}
We apply Property (M1) to $y$ and $w$, to obtain $h,k \in S$
such that
\begin{equation}\label{hRk}
h \, \RRR' \, k \, \RRR' \, kw = hy \, \text{ and } \, L'_y \wedge L'_w = L'_{hy}.  
\end{equation}
We then take $X = hx$ and  $Y = kz$ to give us that
\begin{equation*}
\begin{split}
X\!a &= hxa = hyc = kwc = kze = Y\!e\\
X\!b &= hxb = hyd = kwd = kzf = Y\!f,
\end{split}
\end{equation*}
using (\ref{xayc}), (\ref{hRk}) and (\ref{wcze}).
Also, since $\RRR'$ is a left congruence,
\begin{equation*}
\begin{split}
x \, \RRR' \, xa &\implies X = hx \, \RRR' \, hxa = X\!a\mbox{ and }\\
z \, \RRR' \, ze &\implies Y = kz \, \RRR' \, kze = Y\!e.
\end{split}
\end{equation*}

Using Property (M2) we have that (\ref{hRk}) implies that
$L'_{yc} \wedge L'_{wc} = L'_{hyc}$.
Hence, using $xa = yc$ from (\ref{xayc})
and $wc \, \LLL' \, c$ from (\ref{wcze}),
we have $L'_{xa} \wedge L'_{c} = L'_{hxa}$.
We can then use (\ref{xayc}) to give us $a \, \LLL' \, xa = yc \, \LLL' \, c$, and so
$L'_{a} \wedge L'_{a} = L'_{hxa}$.
Therefore $X\!a \, \LLL' \, a$.

The last relation needed can be obtained similarly or achieved more directly  by noticing that
\begin{equation*}
e \, \LLL' \, c \, \LLL' \, a \, \LLL' \, X\!a = Y\!e.
\end{equation*}
\end{proof}

\begin{lemma}\label{WD}
Multiplication in $P$ is well-defined.
\end{lemma}

\begin{proof}
Let $[a,b],[c,d]\in P$.
From Roadmap \ref{roadmap}, we have that 
\begin{equation*}
{[a,b][c,d] = [ua,vd]}, 
\end{equation*}
where $u, v \in S$ are the elements
that exist by (M1) such that
\begin{equation*}
{\, u \, \RRR' \, v \, \RRR' \, vc = ub}\mbox{ and } L'_b \, \wedge \, L'_c = L'_{ub}.
\end{equation*}

We need to show that the product, $[a,b][c,d]$,
depends neither upon the choice of representative for the equivalence class,
nor the choice of $u$ and $v$ appearing in the rule for multiplication.
We start with the choice of $u$ and $v$.

\textbf{Choice of $u$ and $v$}:
Let
\begin{equation}\label{Choice1}
u \, \RRR' \, v \, \RRR' \, vc = ub\mbox{ and } L'_{b} \, \wedge \, L'_{c} = L'_{ub},
\end{equation}
so that
${[a,b][c,d] = [ua,vd]}$.
Also let
\begin{equation}\label{Choice2}
s \, \RRR' \, t \, \RRR' \, {tc = sb}\mbox{ and }L'_{b} \wedge L'_{c} = L'_{sb},
\end{equation}
so that $[a,b][c,d] = [sa,td]$.
We show that
$(ua,vd) \sim (sa,td)$, which is true exactly if there exists
$w,z \in S$ such that
\begin{equation}\label{Choice3}
wua=zsa, \: wvd=ztd, \; w \, \RRR' \, wua \, \LLL' \, ua\mbox{ and } z \, \RRR' \, zsa \, \LLL' \, sa.
\end{equation}
Applying Property (M1) to $ua$ and $sa$, let $w$ and $z$ be elements such that
\begin{equation}\label{Choice4}
{w \, \RRR' \, z \, \RRR' \, zsa = wua} \, \text{ and } \, {L'_{ua} \wedge L'_{sa} = L'_{wua}}.   
\end{equation}
Using the fact that $a \, \RRR' \, b$, we see that
\begin{equation*}
wua = zsa \overset{(M4)}{\implies} wub=zsb.
\end{equation*}
Then, as $ub=vc$ and $tc=sb$ from (\ref{Choice1}) and (\ref{Choice2}),
this gives us $wvc = ztc$.
We then use $c \, \RRR' \, d$ to obtain
\begin{equation*}
wvc = ztc \overset{(M4)}{\implies} wvd=ztd.
\end{equation*}
From (\ref{Choice1}) and (\ref{Choice2}), we also see that
\begin{equation*}
L'_{ub} = L'_{b} \wedge L'_{c} = L'_{sb} \implies ub \, \LLL'\, sb,
\end{equation*}
which, together with $a \, \RRR' \, b$, implies that $ua \, \LLL'\, sa$ by
Lemma \ref{secprops} (iv).
Using the definition of $\wedge$, along with (\ref{Choice3}), we then have
\begin{equation*}
L'_{ua} = L'_{sa} = L'_{ua} \wedge L'_{sa} = L'_{wua} = L'_{zsa}.
\end{equation*}
This gives us the required properties for $(ua,vd) \sim (sa,td)$.

\textbf{First Variable}:
Let $(a,b) \sim (\tilde{a},\tilde{b})$. Therefore there exists
$x,y \in S$ such that
\begin{equation}\label{FV1}
xa=y\tilde{a}, \, xb=y\tilde{b}, \;
x \, \RRR' \, xa \, \LLL' \, a\mbox{ and } y \, \RRR' \, y\tilde{a} \, \LLL' \, \tilde{a}.
\end{equation}
In order to show well-definedness in the first variable,
we need that for all $[c,d] \in P$, $[a,b][c,d] = [\tilde{a},\tilde{b}][c,d]$.
With that goal in mind,
we apply (M1) to $b$ and $c$, to get that there exists
$u,v \in S$ such that
\begin{equation}\label{FV2}
u \, \RRR' \, v \, \RRR' \, vc = ub \text{ and } L'_{b} \wedge L'_{c} = L'_{ub}.
\end{equation}
Therefore
$[a,b][c,d] = [ua,vd]$.

Our aim is to first find elements $\tilde{u}$ and $\tilde{v}$ which witness
$[\tilde{a},\tilde{b}][c,d]=[\tilde{u}\tilde{a},\tilde{v}d]$.
We will then prove that
$(ua,vd) \sim (\tilde{u}\tilde{a},\tilde{v}d)$.
Of course, we could use (M1) applied to $\tilde{b}$ and $c$,
but for our purposes we need to be more careful.

Applying Property (M1) to $u$ and $x$, we know that there exists
$s,t \in S$ such that
\begin{equation}\label{tusx}
s \, \RRR' \, t \, \RRR' \, {tu = sx}\, \text{ and } \, L'_{x} \wedge L'_{u} = L'_{sx}.
\end{equation}
We take $\tilde{u} = sy$ and $\tilde{v} = tv$.

We want to prove that
$[\tilde{a},\tilde{b}][c,d] = [\tilde{u}\tilde{a},\tilde{v}d]$.
To prove this, it is sufficient that
$\tilde{u} \, \RRR' \, \tilde{v} \, \RRR' \, \tilde{v}c = \tilde{u}\tilde{b}$
and that ${L'_{\tilde{b}} \wedge L'_{c} = L'_{\tilde{u}\tilde{b}}}$.
Substituting letters and rewriting, we need to prove that
\begin{equation}\label{(a1)}
sy \, \RRR' \, tv \, \RRR' \, tvc = sy\tilde{b}
\end{equation}
and
\begin{equation}\label{(b1)}
L'_{\tilde{b}} \wedge L'_{c} = L'_{sy\tilde{b}}.
\end{equation}
We start by proving each relation in Equation (\ref{(a1)}) in turn:

We know that $y \, \RRR' \, x$ from (\ref{FV1})
and $u \, \RRR' \, v$ from (\ref{FV2}).
Using the fact that $\RRR'$ is a left congruence,
$y \, \RRR' \, x$ and $u \, \RRR' \, v$ imply that
$sy \, \RRR' \, sx$ and $tu \, \RRR' \, tv$ respectively.
Then, as $sx=tu$ from (\ref{tusx}), this gives us that $sy \, \RRR' \, tv$.
Using left compatibility of  $\RRR'$ again,
$v \, \RRR' \, vc$ from (\ref{FV2}) implies that $tv \, \RRR' \, tvc$.
Also, using $vc = ub$, $tu = sx$, $xb = y\tilde{b}$, we get
\begin{equation}\label{tvc = sxb}
tvc = tub = sxb = sy\tilde{b}.
\end{equation}

We now verify that Equation (\ref{(b1)}) holds.
We can use (M2) to give us
\begin{equation*}
L'_{x} \wedge L'_{u} = L'_{sx} \implies L'_{xb} \wedge L'_{ub} = L'_{sxb}.
\end{equation*}
Using Lemma \ref{secprops} (ii) we have that $a \, \RRR' \, b$ imples that
$x \, \RRR' \, xb \, \LLL' \, b$.
Using $xb \, \LLL' \, b$ and $L'_{b} \wedge L'_{c} = L'_{ub}$, we have
\begin{equation*}
L'_{b} \wedge (L'_{b} \wedge L'_{c}) = L'_{sxb}.
\end{equation*}
Therefore
\begin{equation}\label{Lsxb}
L'_{b} \wedge L'_{c} = (L'_{b} \wedge L'_{b}) \wedge L'_{c} = L'_{sxb}.
\end{equation} We have that $\tilde{a} \, \RRR' \, \tilde{b}$, so that together with
 $y \, \RRR' \, y\tilde{a} \, \LLL' \, \tilde{a}$ from (\ref{FV1}) 
we may apply Lemma \ref{secprops} (ii) again 
to obtain $y \, \RRR' \, y\tilde{b} \, \LLL' \, \tilde{b}$.
Using (\ref{FV1})  we now have
$b \, \LLL' \, xb = y\tilde{b} \, \LLL' \, \tilde{b}$.
Using this together with $xb = y\tilde{b}$ and (\ref{Lsxb}), we have
\begin{equation*}
L'_{\tilde{b}} \wedge L'_{c} = L'_{sy\tilde{b}},
\end{equation*}
which is (\ref{(b1)}). Therefore $[\tilde{a},\tilde{b}][c,d] = [sy\tilde{a},tvd]$.

Using $xa = y\tilde{a}$ from (\ref{FV1}), this also means that
$[\tilde{a},\tilde{b}][c,d] = [sxa,tvd]$.
Therefore, in order to have well-definedness in the first variable, one needs
${(ua,vd) \sim (sxa, tvd)}$.
This is true exactly if there exists ${w,z \in S}$ such that
\begin{equation*}
wua=zsxa, \, wvd=ztvd, \; w \, \RRR' \, wua \, \LLL' \, ua\mbox{ and } z \, \RRR' \, zsxa \, \LLL' \, sxa.
\end{equation*}
Applying Property (M1) to $ua$ and $sxa$, take $w$ and $z$ to be elements in $S$ such that
\begin{equation}\label{Lwua}
w \, \RRR' \, z \, \RRR' \, {zsxa = wua} \,
\text{ and } \, L'_{ua} \wedge L'_{sxa} = L'_{wua}.    
\end{equation}
Since $a \, \RRR' \, b$, we know that
${wua = zsxa}$ implies $wub = zsxb$ by (M4).
We then use $sxb = tvc$ from (\ref{tvc = sxb}) and
$ub = vc$ from (\ref{FV2}) to obtain $wvc = ztvc$.
And therefore, using (M4) again, $c \, \RRR' \, d$ implies that $wvd = ztvd$.

Using Property (M2),
$L'_{x} \wedge L'_{u} = L'_{sx}$ implies that $L'_{xa} \wedge L'_{ua} = L'_{sxa}$.
We then use $xa \, \LLL' \, a$ from (\ref{FV1})
and Lemma \ref{secprops} (iii), to get
\begin{equation*}
L'_{xa} \wedge L'_{ua} = L'_{sxa} \implies L'_{a} \wedge L'_{ua} = L'_{sxa}
\implies L'_{ua} = L'_{sxa}.
\end{equation*}
We then compare this with $L'_{ua} \wedge L'_{sxa} = L'_{wua}$
from (\ref{Lwua}) to
obtain $L'_{ua} = L'_{wua}$.
Lastly, $sxa \, \LLL' \, ua \, \LLL' \, wua = zsxa$.
Altogether, this proves that \[{(ua,vd) \sim (sxa, tvd)} = (\tilde{u}\tilde{a},\tilde{v}d).\]

\textbf{Second Variable}:
Let $(c,d) \sim (\tilde{c},\tilde{d})$. Therefore there exists
$x,y \in S$ such that
\begin{equation}\label{SV1}
xc=y\tilde{c}, \, xd=y\tilde{d}, \;
x \, \RRR' \, xc \, \LLL' \, c\mbox{ and } y \, \RRR' \, y\tilde{c} \, \LLL' \, \tilde{c}.
\end{equation}
In order to show well-definedness in the second variable,
we need that for all $[a,b] \in P$ we have the equality $[a,b][c,d] = [a,b][\tilde{c},\tilde{d}]$.
With that goal in mind, given $[a,b] \in P$,
we apply (M1) to $b$ and $c$, to get that there
exists $u,v \in S$ such that
\begin{equation}\label{SV2}
u \, \RRR' \, v \, \RRR' \, vc = ub \,
\text{ and } \, L'_{b} \wedge L'_{c} = L'_{ub},
\end{equation}
so that $[a,b][c,d] = [ua,vd]$.

Our aim is to find elements $\tilde{u}$ and $\tilde{v}$ which witness
$[a,b][\tilde{c},\tilde{d}]=[\tilde{u}a,\tilde{v}\tilde{d}]$.
We will then prove that
$(ua,vd) \sim (\tilde{u}a,\tilde{v}\tilde{d})$.

Applying Property (M1) to $v$ and $x$, we know that there exists
$p,q \in S$ such that
\begin{equation}\label{pvqx}
p \, \RRR' \, q \, \RRR' \, {qx = pv} \, \text{ and } \, {L'_{v} \wedge L'_{x} = L'_{pv}}.  
\end{equation}
We take $\tilde{u} = pu$ and $\tilde{v} = qy$.

We want to prove that
$[a,b][\tilde{c},\tilde{d}]=[\tilde{u}a,\tilde{v}\tilde{d}]$.
To prove this, it is sufficient that
\begin{equation}
\tilde{u} \, \RRR' \, \tilde{v} \, \RRR' \, \tilde{v}\tilde{c} = \tilde{u}b
\, \text{ and } \, {L'_{b} \wedge L'_{\tilde{c}} = L'_{\tilde{u}b}}.   
\end{equation}
Rewriting this, we need to prove that
\begin{equation}\label{(a2)}
pu \, \RRR' \, qy \, \RRR' \, qy\tilde{c} = pub
\end{equation}
and
\begin{equation}\label{(b2)}
L'_b \wedge L'_{\tilde{c}} = L'_{pub}.
\end{equation}
We start be proving each relation in Equation (\ref{(a2)}) in turn:

We know that $u \, \RRR' \, v$ from (\ref{SV2})
and that $x \, \RRR' \, y$ from (\ref{SV1}).
Using the fact that $\RRR'$ is a left congruence,
$u \, \RRR' \, v$ and $x \, \RRR' \, y$ imply that
$pu \, \RRR' \, pv$ and $qx \, \RRR' \, qy$ respectively.
Then, since $pv = qx$ from (\ref{pvqx}), this gives us that $pu \, \RRR' \, qy$.
Using left compatibility of $\RRR'$ again,
$u \, \RRR' \, ub$ from (\ref{SV2}) implies that $pu \, \RRR' \, pub$.
Also, using $ub = vc$ from (\ref{SV2}), $pv = qx$ from (\ref{pvqx}),
and $xc = y\tilde{c}$ from (\ref{SV1}), we get
\begin{equation}\label{pub = qxc}
pub = pvc = qxc = qy\tilde{c}.
\end{equation}
We now prove Equation (\ref{(b2)}).
Using (M2) applied to (\ref{pvqx}), we have
\begin{equation*}
L'_{v} \wedge L'_{x} = L'_{pv} \implies
L'_{vc} \wedge L'_{xc} = L'_{pvc}.
\end{equation*}
We use $L'_{b} \wedge L'_{c} = L'_{vc}$ from (\ref{SV2})
and $xc \, \LLL' \, c$ from (\ref{SV1})
to get
\begin{equation*}
(L'_{b} \wedge L'_{c}) \wedge L'_{c} = L'_{pvc}.
\end{equation*}
Therefore
\begin{equation*}
L'_{b} \wedge L'_{c} = L'_{b} \wedge (L'_{c} \wedge L'_{c}) = L'_{pvc}.
\end{equation*}
We apply $c \, \LLL' \, xc = y\tilde{c}\, \LLL' \, \tilde{c}$
from (\ref{SV1}) and $ub = vc$ from (\ref{SV2}) to get
\begin{equation*}
L'_{b} \wedge L'_{\tilde{c}} = L'_{pub}.
\end{equation*}
This concludes the verification of Equations (\ref{(a2)}) and (\ref{(b2)}), so that 
using (\ref{SV1})
\[[a,b][\tilde{c},\tilde{d}] = [pua,qy\tilde{d}] = [pua,qxd].\]

Therefore, in order to have well-definedness in the second variable, one needs
$(ua,vd) \sim (pua, qxd)$.
This is true exactly if there exists $w,z \in S$ such that
\begin{equation*}
wua=zpua, \, wvd=zqxd, \; w \, \RRR' \, wua \, \LLL' \, ua\mbox{ and } z \, \RRR' \, zpua \, \LLL' \, pua.
\end{equation*}
Applying Property (M1) to $ua$ and $pua$, take $w$ and $z$ to be elements in $S$ such that
$w \, \RRR' \, z \, \RRR' \, {zpua = wua}$ and $L'_{ua} \wedge L'_{pua} = L'_{wua}$.
We use $pub = qxc$ from (\ref{pub = qxc})
and $ub = vc$ from (\ref{SV2}), along with
$a \, \RRR' \, b$ and $c \, \RRR' \, d$, to obtain
\begin{equation*}
wua = zpua \overset{(M4)}{\implies}
wub = zpub \implies wvc = zqxc \overset{(M4)}{\implies} wvd = zqxd.
\end{equation*}
Using Property (M2),
$L'_{v} \wedge L'_{x} = L'_{pv}$ implies that $L'_{vc} \wedge L'_{xc} = L'_{pvc}$.
We then use $xc \, \LLL' \, c$ from (\ref{SV1})
and Lemma \ref{secprops} (iii), to get
\begin{equation*}
L'_{vc} \wedge L'_{xc} = L'_{pvc} \implies L'_{vc} \wedge L'_{c} = L'_{pvc}
\implies L'_{vc} = L'_{pvc} \implies L'_{ub} = L'_{pub}.
\end{equation*}
We apply Lemma \ref{secprops} (iv) to $ub \, \LLL' \, pub$ and
$a \, \RRR' \, b$ to obtain $ua \, \LLL' \, pua$.
Therefore $L'_{wua} = {L'_{ua} \wedge L'_{pua}} = L'_{ua}$.
Finally,  $pua \, \LLL' \, ua \, \LLL' \, wua = zpua$,
which gives us all the
necessary conditions for ${(ua,vd) \sim (pua, qxd) = (\tilde{u}a, \tilde{v}\tilde{d})}$.

\vspace{2mm}

Note that by using well-definedness in the first variable and well-definedness in the second variable
together, we can see that for  $(a,b) \sim (\tilde{a},\tilde{b})$ and  $(c,d) \sim (\tilde{c},\tilde{d})$,
we get
\begin{equation*}
(a,b)(c,d) \sim (\tilde{a},\tilde{b})(c,d) \sim (\tilde{a},\tilde{b})(\tilde{c},\tilde{d}).
\end{equation*}
Therefore, by transitivity, this multiplication is well-defined.
\end{proof}

\begin{lemma}\label{Ass}
Multiplication in $P$ is associative.
\end{lemma}

\begin{proof}
Let $[a,b],[c,d],[e,f] \in P$.

Applying Property (M1) to $b$ and $c$, we choose $u,v \in S$ satisfying
\begin{equation}
u \, \RRR' \, v \, \RRR' \, {vc = ub} \, \text{ and } \, L'_{b} \wedge L'_{c} = L'_{ub}. 
\end{equation}
This gives us that $[a,b][c,d] = [ua,vd]$.
Similarly, we choose $p,q \in S$ satisfying
\begin{equation}
p \, \RRR' \, q \, \RRR' \, qe = pd \, \text{ and } \, L'_{d} \wedge L'_{e} = L'_{pd}. 
\end{equation}
Then $[c,d][e,f] = [pc,qf]$.

Applying Property (M1) to $v$ and $p$, we know that there exists
$i,j \in S$ such that
\begin{equation}
i \, \RRR' \, j \, \RRR' \, {jp = iv} \, \text{ and } \, {L'_{v} \wedge L'_{p} = L'_{iv}}.   
\end{equation}

We want to prove that
\begin{equation} \label{Ass1}
([a,b][c,d])[e,f] = [ua,vd][e,f] = [i(ua),(jq)f],
\end{equation}
and that
\begin{equation} \label{Ass2}
[a,b]([c,d][e,f]) = [a,b][pc,qf] = [(iu)a,j(qf)].
\end{equation}
This would prove associativity.

In order to prove (\ref{Ass1}), we need
\begin{equation}\label{(a3)}
i \, \RRR' \, jq \, \RRR' \, {jqe = ivd}
\end{equation}
and
\begin{equation}\label{(b3)}
L'_{vd} \wedge L'_{e} = L'_{ivd}.
\end{equation}
We start by proving each relation in Equation (\ref{(a3)}).

Since $\RRR'$ is a left congruence, $q \, \RRR' \, p$ implies that
$jq \, \RRR' \, jp$, which in turn is $\RRR'$-related to $i$.
Using again the left compatibility of $\RRR'$, we see that
$q \, \RRR' \, qe$ implies that $jq \, \RRR' \, jqe$.
Also, using $qe = pd$ and $jp = iv$, we see that $jqe = jpd = ivd$.

We now prove Equation (\ref{(b3)}).
We apply (M2) to
$L'_{v} \wedge L'_{p} = L'_{iv}$ to give us that
$L'_{vd} \wedge L'_{pd} = L'_{ivd}$.
And so, using $L'_{d} \wedge L'_{e} = L'_{pd}$ and
Lemma \ref{secprops} (iii), we have
\begin{equation*}
\begin{split}
L'_{vd} \, \wedge \, L'_{pd} = L'_{ivd}
&\implies L'_{vd} \, \wedge \, (L'_{d} \, \wedge \, L'_{e}) = L'_{ivd}\\
&\implies (L'_{vd} \, \wedge \, L'_{d}) \, \wedge \, L'_{e} = L'_{ivd}\\
&\implies L'_{vd} \, \wedge\,  L'_{e} = L'_{ivd}.
\end{split}
\end{equation*}
We now have proved both (\ref{(a3)}) and (\ref{(b3)}),
which together gives us (\ref{Ass1}).

In order to prove (\ref{Ass2}), we need
\begin{equation}\label{(c3)}
iu \, \RRR' \, j \, \RRR' \, {jpc = iub}
\end{equation}
and
\begin{equation}\label{(d3)}
L'_{b} \wedge L'_{pc} = L'_{iub}.
\end{equation}
We start by proving each relation in Equation (\ref{(c3)}).

Since $\RRR'$ is a left congruence, $u \, \RRR' \, v$ implies that
$iu \, \RRR' \, iv$, which in turn is $\RRR'$-related to $j$.
Using again the left compatibility of $\RRR'$, we see that
$c \, \RRR' \, d$ implies that $jpc \, \RRR' \, jpd$
and $p \, \RRR' \, pd$ implies that $jp \, \RRR' \, jpd$.
Therefore $j \, \RRR' \, jp \, \RRR' \, jpd \, \RRR' \, jpc$.
Also, using $ub = vc$ and $iv = jp$, we see that $iub = ivc = jpc$.

We now prove Equation (\ref{(d3)}).
We apply (M2) to
$L'_{v} \wedge L'_{p} = L'_{iv}$ to give us that
$L'_{vc} \wedge L'_{pc} = L'_{ivc}$.
And so, using $L'_{b} \wedge L'_{c} = L'_{vc}$ and
Lemma \ref{secprops} (iii), we have
\begin{equation*}
\begin{split}
L'_{vc} \, \wedge \, L'_{pc} = L'_{ivc}
&\implies (L'_{b} \, \wedge \, L'_{c}) \, \wedge L'_{pc} = L'_{iub}\\
&\implies L'_{b} \, \wedge \, (L'_{c} \, \wedge \, L'_{pc}) = L'_{iub}\\
&\implies L'_{b} \, \wedge \, L'_{pc} = L'_{iub}.
\end{split}
\end{equation*}
We have now proved both (\ref{(c3)}) and (\ref{(d3)}),
which together gives us (\ref{Ass2}), finishing the proof.
\end{proof}

We have now proved that $P$ is a semigroup.
The next lemma provides a couple of
useful shortcuts to help in later stages.

\begin{lemma}\label{extras}
These following statements are true in $P$:
\begin{enumerate}[(i)]
    \item $[a,a] = [b,b]$ if and only if $a  \, \LLL' \, b$;
    \item $[a,b][b,a] = [a,a]$.
\end{enumerate}
\end{lemma}

\begin{proof} $(i)$  We know that $[a,a] = [b,b]$
if and only if there exists $w,z \in S$ such that
\begin{equation} \label{aabb}
wa=zb, \; w \, \RRR' \, wa \, \LLL' \, a\mbox{ and } z \, \RRR' \, zb \, \LLL' \, b.
\end{equation}
Let $[a,a] = [b,b]$. Therefore there exists $w,z \in S$ satisfying (\ref{aabb}).
Hence $a \, \LLL' \, wa = zb \, \LLL' \, b$.

Conversely let $a  \, \LLL' \, b$.
Applying (M1) to $a$ and $b$, there exists $w, z \in S$ such that
$w \, \RRR' \, z \, \RRR' \, {zb = wa}$ and
${L'_{a} \wedge L'_{b} = L'_{wa}}$
Therefore, since $a  \, \LLL' \, b$,
we have $ L'_{wa} = L'_{a} \wedge L'_{b} = L'_{a}$, and consequently
$zb = w a\, \LLL' \, a \, \LLL' \, b$.
Comparing with (\ref{aabb}), we see that ${[a,a] = [b,b]}$.

$(ii)$ By Lemma \ref{secprops} (i), there exists $y \in S$ such that
$y \, \RRR' \, yb \, \LLL' \, b$.
By comparing with the definition of multiplication
in Roadmap \ref{roadmap}, we see that ${[a,b][b,a] = [ya,ya]}$.
Since $a \, \RRR' \, b$,
Lemma \ref{secprops} (ii) gives us that $\, y \, \RRR' \, ya \, \LLL' \, a$.
So by (i), $[a,b][b,a] = [a,a]$.
\end{proof}

\begin{lemma}\label{regular}
The semigroup $P$ is regular.
\end{lemma}

\begin{proof}
Let $[a,b] \in P$.
By Lemma \ref{extras} (ii), $[a,b][b,a][a,b] = [a,a][a,b]$.
By Lemma \ref{secprops} (i), there exists $y \in S$ such that
$y \, \RRR' \, ya \, \LLL' \, a$.
Therefore, by our definition of multiplication,
$[a,a][a,b] = [ya,yb]$.

We want to prove that $(a,b) \sim (ya,yb)$. That is there exists $w,z \in S$ such that
\begin{equation*}
wa=zya, \, wb=zyb, \; w \, \RRR' \, wa \, \LLL' \, a\mbox{ and } z \, \RRR' \, zya \, \LLL' \, ya.
\end{equation*}
Applying Property (M1) to $a$ and $ya$, there exists $w,z \in S$ such that
\begin{equation*}
w \, \RRR' \, z \, \RRR' \, {zya = wa} \, \text{ and } \, {L'_{a} \wedge L'_{ya} = L'_{wa}}.
\end{equation*}
We use (M4) to give us that $wa = zya$ implies $wb=zyb$.
We can also use $a \, \LLL' \, ya$ to give us that
$L'_{wa} = L'_{a} \, \wedge \, L'_{ya} = L'_{a}$.
Therefore ${a \, \LLL' \, ya \, \LLL' \, wa = zya}$.

So we have that $[a,b][b,a][a,b] = [a,b]$. Therefore $P$ is regular.
\end{proof}

It is good to note that in the exactly same way $[b,a][a,b][b,a] = [b,a]$.
Therefore $[b,a] \in V([a,b])$.

\begin{lemma}\label{Inv}
The semigroup $P$ is an inverse semigroup, with semilattice of idempotents 
\[E(P)=\{ [a,a]:a\in S\}\]
and the inverse of $[a,b]$  given by $[a,b]^{-1} = [b,a]$.
\end{lemma}

\begin{proof}
We start by identifying the idempotents of $P$.
For any $[a,a]\in P$ it is easy to see, using (M1), that
$[a,a]\in E(P)$.

Let $[a,b] \in P$ be an idempotent, i.e. let $[a,b][a,b]=[a,b]$.
We know that ${[a,b][a,b] = [ua,vb]}$,
where $u$ and $v$ are the elements that exists by (M1) such that
\begin{equation*}
u \, \RRR' \, v \, \RRR' \, va = ub \, \text{ and }
\,{L'_{a} \wedge L'_{b} = L'_{ub}}\, .
\end{equation*}
Consequently we know that $(a,b) \sim (ua,vb)$. Therefore there exists $w,z \in S$ such that
\begin{equation*}
wa=zua, \, wb=zvb, \; w \, \RRR' \, wa \, \LLL' \, a, \, z \, \RRR' \, zua \, \LLL' \, ua.
\end{equation*}
By Lemma \ref{secprops} (ii) we have that
$a \, \RRR' \, b$ implies that
$w \, \RRR' \, wb \, \LLL' \, b$.
Also, by applying (M4) to both $wb=zvb$ and $wa=zua$
and using $va = ub$, we have
\begin{equation*}
wa \overset{(M4)}{=} zva = zub \overset{(M4)}{=} wb.
\end{equation*}
Therefore $a \, \LLL' \, wa = wb \, \LLL' \, b$. 
We then apply Property (M6) to give us $a=b$.
Therefore $E(P)$  is as given in the statement. 

We now prove that idempotents commute.
Let $[a,a], [b,b]\in E(P)$. Applying Property (M1) to $a$ and $b$,
we choose $u$ and $v$ such that
${[a,a][b,b] = [ua,vb]}$, where
\begin{equation*}
u \, \RRR' \, v \, \RRR'\, vb = ua \, \text{ and } \,
{L'_{a} \wedge L'_{b} = L'_{ua}}.
\end{equation*}
By inspection we can see that $v$ and $u$ satisfy the necessary properties for
${[b,b][a,a] = [vb,ua]}$.
And so we see that, since $ua = vb$, we have
\begin{equation*}
[a,a][b,b] = [ua,vb] = [vb,ua] = [b,b][a,a].
\end{equation*}
Therefore the idempotents of $P$ commute.
Since $P$ is also regular, this means that $P$ is an inverse semigroup.

Moreover since $[b,a] \in V([a,b])$, we easily see that $[a,b]^{-1} = [b,a]$ for all ${[a,b] \in P}$.
\end{proof}

We now prove that $S$ embeds into $P$.
We do this by defining a function $\phi :S \rightarrow P$, by
$a \phi = [x,xa]$, where $x$ is the element such that $x \, \RRR' \, xa \, \LLL' \, a$,
that exists by Lemma \ref{secprops} (i).
Note that $[x,xa] \in P$.

\begin{lemma}\label{EB}
The function $\phi$ is an embedding.
\end{lemma}

\begin{proof} Again we must proceed in stages.

\textbf{Well-defined}:
Let $x \, \RRR' \, xa \, \LLL' \, a$ and let $y \, \RRR' \, ya \, \LLL' \, a$.
By our definition, this means that ${a \phi = [x,xa]}$ and that ${a \phi = [y,ya]}$.
Therefore, in order to prove that $\phi$ is well-defined, we need to prove that
$(x,xa) \sim (y,ya)$.
This is true exactly if there exists $w,z \in S$ such that
\begin{equation*}
wx=zy, \, wxa=zya, \; w \, \RRR' \, wx \, \LLL' \, x\mbox{ and } z \, \RRR' \, zy \, \LLL' \, y.
\end{equation*}
Applying Property (M1) to $x$ and $y$, we take $w$
and $z$ to be elements in $S$ such that
${w \, \RRR' \, z \, \RRR' \, wx = zy}$ and ${L'_{x} \wedge L'_{y} = L'_{wx}} \,$.
Trivially, $wxa = zya$.
Using (M5),
$x \, \RRR' \, xa \, \LLL' \, a$ and $y \, \RRR' \, ya \, \LLL' \, a$
implies that $x \, \LLL' y$.
Therefore
\begin{equation*}
L'_{wx} = L'_{x} \wedge L'_{y} = L'_{x}.
\end{equation*}
For the last necessary property, we notice
$y \, \LLL' \, x \, \LLL' \, wx = zy$.

\textbf{Homomorphism}:
Let $a,b \in S$, and let $x,y \in S$ such that
$x \, \RRR' \, xa \, \LLL' \, a$ and
$y \, \RRR' \, yb \, \LLL' \, b$.
Therefore, by definition,
$a \phi = [x,xa]$ and $b \phi = [y,yb]$. Then
\begin{equation*}
(a \phi) (b \phi) = [x,xa][y,yb] = [ux,vyb] = [ux,uxab],
\end{equation*}
where $u$ and $v$ are the elements that exist by (M1) such that
\begin{equation*}
u \, \RRR' \, v \, \RRR' \, vy = uxa \, \text{ and }
\,{L'_{xa} \wedge L'_{y} = L'_{uxa}}.
\end{equation*}
We want to prove that this is equal to $(ab) \phi$.

Using the fact that $\RRR'$ is a left conguence,
we have that $yb \, \RRR' \, y$ implies that $vyb \, \RRR' \, vy$,
and $xa \, \RRR' \, x$ implies that $uxa \, \RRR' \, ux$.
Therefore
\begin{equation*}
uxab = vyb \, \RRR' \, vy = uxa \, \RRR' \, ux. 
\end{equation*}
We use $xa \, \LLL' \, a$ to obtain ${L'_{a} \wedge L'_{y} = L'_{uxa}}$.
We can then apply Property (M2) to
${L'_{a} \wedge L'_{y} = L'_{uxa}}$
to give us that ${L'_{ab} \wedge L'_{yb} = L'_{uxab}}$.
Using $yb \, \LLL' \, b$, this means that
${L'_{ab} \wedge L'_{b} = L'_{uxab}}$.
We can then apply Lemma \ref{secprops} (iii) to give us $ab \, \LLL' \, uxab$.

By the definition of $\phi$, since $ux \, \RRR' \, uxab \, \LLL' \, ab$,
this means that
\begin{equation*}
(ab)\phi = [ux,uxab] = (a \phi) (b \phi).   
\end{equation*}

\textbf{Injective}:
Let $a,b \in S$ such that $a \phi = b \phi$.
Therefore, choosing $x$ and $y$ such that
$x \, \RRR' \, xa \, \LLL' \, a$ and $y \, \RRR' \, yb \, \LLL' \, b$,
we have that $[x,xa] = [y,yb]$.
This means there exists $w,z \in S$
such that
\begin{equation*}
wx=zy, \, wxa=zyb, \; w \, \RRR' \, wx \, \LLL' \, x
\text{ and } z \, \RRR' \, zy \, \LLL' \, y.
\end{equation*}
Therefore, using the fact that $\LLL'$ is a right congruence,
we have that $x \, \LLL' \, wx$ implies that $xa \, \LLL' \, wxa$.
Consequently, $a \, \LLL' \, xa \, \LLL' \, wxa$.
Similarly, $y \, \LLL' \, zy$ implies that $yb \, \LLL' \, zyb$.
And so, $b \, \LLL' \, yb \, \LLL' \, zyb = wxb$,
using $zy = wx$ in the last equality.
Therefore, we can apply Property (M6) giving us that
$a \, \LLL' \, wxa = wxb \, \LLL' \, b$ implies that $a=b$.
\end{proof}

\begin{lemma}\label{Rproof}
Let $a,b \in S$. Then $a \, \RRR' \, b$ if and only if $a \phi \, \RRR^P \, b \phi$.
\end{lemma}

\begin{proof}
We have already proved that $P$ is an inverse semigroup,
so $a \phi \, \RRR^P \, b \phi$ if and only if
${(a \phi)(a \phi)^{-1} = (b \phi)(b \phi)^{-1}}$.

Let $x \, \RRR' \, xa \, \LLL' \, a$
and $y \, \RRR' \, yb \, \LLL' \, b$,
so that $a \phi = [x,xa]$ and $b \phi = [y,yb]$.
Then $(a \phi)(a \phi)^{-1} = [x,xa][xa,x] = [x,x]$, by Lemmas
\ref{Inv} and \ref{extras} (ii).
Similarly $(b \phi)(b \phi)^{-1} = [y,yb][yb,b] = [y,y]$.

Therefore
$a \phi \, \RRR^P \, b \phi$ if and only if $(x,x) \sim (y,y)$,
which is true if and only if
$x \, \LLL' \, y$, using Lemma \ref{extras} (i).
We then use (M5) to give us that 
this is equivalent to $a \, \RRR' \, b$.
\end{proof}

\begin{lemma}\label{Lproof}
Let $a,b \in S$. Then $a \leq_l b$ if and only if
$a \phi \leq_{\LLL^P} b \phi$.
\end{lemma}

\begin{proof}
We have already proved that $P$ is an inverse semigroup,
so $a \phi \leq_{\LLL^P} b \phi$ if and only if
${a \phi = (a \phi)(b \phi)^{-1}(b \phi)}$.

Let $x \, \RRR' \, xa \, \LLL' \, a$ and $y \, \RRR' \, yb \, \LLL' \, b$,
so that $a \phi = [x,xa]$ and $b \phi = [y,yb]$.
Using Lemma \ref{extras}, we have
$(b \phi)^{-1}(b \phi) = [yb,y][y,yb] = [yb,yb] = [b,b]$.
Therefore
\begin{equation*}
(a \phi)(b \phi)^{-1}(b \phi) = [x,xa][b,b] = [ux,vb],
\end{equation*}
where $u$ and $v$ are the elements that exist by (M1) such that
\begin{equation*}
u \, \RRR' \, v \, \RRR' \, vb = uxa \,
\text{ and } \, {L'_{xa} \wedge L'_{b} = L'_{uxa}}.
\end{equation*}
Note that since $xa \, \LLL \, a$, this means that ${L'_{a} \wedge L'_{b} = L'_{uxa}}$.

We use $vb = uxa$ to give us that
$(a \phi)(b \phi)^{-1}(b \phi) = [ux,uxa]$.
It follows that  $a \phi \leq_{\LLL^P} b \phi$ if and only if
$(x,xa) \sim (ux,uxa)$, which is true exactly if there exists $w,z \in S$ such that
\begin{equation}\label{preasterix}
wx = zux, \, wxa = zuxa, \; w \, \RRR' \,  wx \, \LLL' \, x\mbox{ and }
z \, \RRR' \, zux \, \LLL' \, ux.
\end{equation}
We know that $x \, \RRR' \, xa$, and so $ux \, \RRR' \, uxa$ as $\RRR'$ is a left congruence.
Therefore we can use Lemma \ref{secprops} (ii)
to rewrite (\ref{preasterix}) to the equivalent expression (\ref{asterix}).
That is, $a \phi \leq_{\LLL^P} b \phi$ if and only if
there exists $w,z \in S$ such that
\begin{equation}\label{asterix}
wx = zux, \, wxa = zuxa, \; w \, \RRR' \,  wxa \, \LLL' \, xa\mbox{ and }
z \, \RRR' \, zuxa \, \LLL' \, uxa.
\end{equation}

Let $a \phi \leq_{\LLL^P} b \phi$, i.e. let $w$ and $z$ exist in $S$
such that (\ref{asterix}) is satisfied.
We see that $uxa \, \LLL' \, zuxa = wxa \, \LLL' \, xa \, \LLL' \, a$.
Therefore ${L'_{a} \wedge L'_{b} = L'_{uxa} = L'_{a}}$.
By definition this means that $a \leq_l b$.

On the other hand, let $a \leq_l b$. By definition
${L'_{a} \wedge L'_{b} = L'_{a}}$.
Therefore
\begin{equation*}
{L'_{uxa} = L'_{a} \wedge L'_{b} = L'_{a} = L'_{xa}}.
\end{equation*}
Applying Property (M1) to $xa$ and $uxa$, there exists $w,z \in S$ such that
\begin{equation*}
w \, \RRR' \, z \, \RRR' \, {zuxa = wxa} \text{ and } {L'_{uxa} \wedge L'_{xa} = L'_{wxa}}.
\end{equation*}
Using $x \, \RRR' \, xa$, we know that $zuxa = wxa$ implies that $zux = wx$ by (M4).
Using the fact that $uxa \, \LLL' \, xa$, we see that
$L'_{wxa} = L'_{uxa} \wedge L'_{xa} = L'_{xa}$.
Therefore $uxa \, \LLL' \, xa \, \LLL' \, wxa = zuxa$.
This gives us (\ref{asterix}), and so
$a \phi \leq_{\LLL^P} b \phi$.
\end{proof}

\begin{lemma}\label{ssliq}
The semigroup $P$ is a semigroup of straight left I-quotients of $S \phi$.
\end{lemma}

\begin{proof}
Let $[a,b] \in P$. Note that $a,b \in S$ with $a \, \RRR' \, b$.

Let $x \, \RRR' \, xa \, \LLL' \, a$
and $y \, \RRR' \, yb \, \LLL' \, b$,
so that $a  \phi = [x,xa]$ and $b  \phi = [y,yb]$.
By \mbox{Lemma \ref{Rproof}}, $a \phi \, \RRR^P \, b \phi$.
We have
\begin{equation*}
(a \phi)^{-1}(b \phi) = [xa,x][y,yb] = [uxa,vyb],
\end{equation*}
where $u$ and $v$ are the elements that exist by (M1) such that
\begin{equation*}
u \, \RRR' \, v \, \RRR' \, vy = ux \, \text{ and }
\, {L'_{x} \wedge L'_{y} = L'_{ux}}.
\end{equation*}
We want to prove that $(a \phi)^{-1}(b \phi) = [a,b]$.

We see that $(a,b) \sim (uxa,vyb)$ exactly if there exists $w,z \in S$ such that
\begin{equation*}
wa=zuxa, \, wb=zvyb, \; w \, \RRR' \,  wa \, \LLL' \, a\mbox{ and }
z \, \RRR' \, zuxa \, \LLL' \, uxa.
\end{equation*}
Applying Property (M1) to $a$ and $uxa$,
we know that there exists $w,z \in S$ such that
\begin{equation*}
w \, \RRR' \, z \, \RRR' \, {zuxa = wa} \, \text{ and } \,
{L'_{a} \wedge L'_{uxa} = L'_{wa}}.
\end{equation*}
We see that $wa=zuxa$ implies ${wb=zuxb}$ by (M4),
and therefore, since $ux = vy$ we have $wb = zvyb$.
We use Property (M5) to get that
$a \, \RRR' \, b$ implies $x \, \LLL' y$, and therefore
${L'_{ux} = L'_{x} \wedge L'_{y} = L'_{x}}$.
We then use the fact that $\LLL'$ is a right congruence to give us that
$ux \, \LLL' \, x$ implies $uxa \, \LLL' \, xa \, \LLL' \, a$.
Therefore ${L'_{wa} = L'_{a} \wedge L'_{uxa} = L'_{a}}$, and so
$zuxa = wa \, \LLL' \, a \, \LLL' \, uxa$.
This gives us ${[a,b] = (a \phi)^{-1}(b \phi)}$, where $a \phi \, \RRR^P \, b \phi$.
\end{proof}

The  proof of Theorem \ref{main} is now complete. 

\section{Right ample straight left I-orders}\label{sec:ample}

The aim of this section is to give two  applications
of Theorem~\ref{main} to new classes of semigroups. We could, of course, apply it to describe left I-orders in some classes of semigroups, such as primitive inverse semigroups, that have already been considered, but we refer the reader to the thesis of the second author \cite{schneiderthesis} for those arguments.  

In this work we have already seen the importance of the relation $\ars$, introduced in Section~\ref{sec:themainresult}. The dual relation is denoted by $\els$. In fact,  $\els$ is the (right compatible) equivalence relation associated with the (right compatible) preorder $\leqels$, where for $a,b\in S$ we have that $a\,\leqels\, b$ if and only if
for all $x,y\in S^1$, if $bx=by$ then $ax=ay$. The latter condition is equivalent to $a\, \leqel\, b$ in some oversemigroup of $S$. Of course a dual statement is true for $\ars$ but we do not explicitly need that here. We remark that as for $\ars$, the relations $\els$ and $\leqels$ will always refer to $S$ itself.

A semigroup for which $E(S)$ is a semilattice and every element is $\ars$-related to an idempotent is said to be {\em left adequate}. {\em Right adequate} semigroups are defined dually,
and a semigroup is {\em adequate} if it is right and left adequate. Notice that if $S$ is left (right) adequate, then $a\in S$ is
$\ars$-related ($\els$-related) to a unique idempotent, which we denote by $a^+$ ($a^*$). Right, left and (two-sided) adequate semigroups are potential analogues of inverse semigroups. It transpires that for the closest connections with inverse semigroups, we need some control on moving idempotents in products, as witnessed by the following definition.

\begin{definition}\label{defn:ample} A left adequate semigroup is {\em left ample} if for all $a,b\in S$ we have
\[ab^+=(ab)^+a.\]
{\em Right ample} semigroups are defined dually, and a semigroup is {\em ample} if it is both right  and left ample. 
\end{definition}
 
 It is easy to see that an inverse semigroup is ample, with
 $a^+=aa^{-1}$ and $a^*=a^{-1}a$. Moreover, if $S$ is a subsemigroup of an inverse semigroup $Q$ in a way that preserves $^+$, that is, it is as a {\em unary subsemigroup} of $Q$, then $S$ is left ample, with the dual and two-sided statements also holding. The precise relationship between one- and  two-sided ample semigroups and inverse semigroups is complex. 
It is worth noting that an ample semigroup $S$
may not be embeddable into an inverse semigroup in such a way
that preserves both $^+$ and $^*$. We refer the reader to \cite{gould2005faithful} for further details.

In \cite{ghroda2010bisimple} a characterisation was given of a particular kind of {\em left} ample straight left I-orders. 
In this section, we 
characterise {\em right} ample
and ({\em two-sided}) ample straight left I-orders,
sitting as unary and bi-unary subsemigroups of their semigroups of I-quotients, 
as an application of Theorem \ref{main}.
Note that not every right ample left I-order is straight. It is easy to see that the left I-order given in  Example \ref{nonstraight} is a left cancellative monoid, which is easily seen to be right ample.



We start with four useful lemmas. The first follows from the well-known fact that the restriction of $\leqel$ to regular elements is exactly $\leqels$.

\begin{lemma}\label{leqsame}
Let $S$ be a semigroup  and let ${e,f \in E(S)}$.
Then  $e \leq_{\LLL} f$ if and only if $e \leq_{\LLL^*} f$.
\end{lemma}

\begin{lemma}\label{wedgestar}
Let $S$ be a semigroup such that $E(S)$ is a semilattice. Then
the $\els$-classes of idempotents form a meet semilattice under the induced partial order on $S/\els$. Moreover,  for  ${e,f \in E(S)}$ we have
 $L^*_e \, \wedge \, L^*_f = L^*_{ef}$.
\end{lemma}

\begin{proof} Since $\leqel\subseteq \leqels$ we certainly have $ef \leq_{\LLL^*} e,f$.

Now let $h \in S$ such that $h \leq_{\LLL^*} e,f$.
As $e1=ee$ we have $h=h1=he$ and similarly $h=hf$, so that
$h=hef\,\leqel\, ef$ and hence $h\, \leqels\, ef$. The result follows. 
\end{proof}

We give some elementary properties of right ample semigroups.
The duals of these properties apply to left ample semigroups. 

\begin{lemma}\label{right ample properties}
Let $S$ be a right ample semigroup.
Then for all $a,b\in S$:
\begin{enumerate}[(i)]
\item $aa^*=a$;
\item $(ab)^*=(a^*b)^*$;
\item $b^* = (ab)^*$ if and only if $b = a^*b$.
\end{enumerate}
\end{lemma}

\begin{proof} Claim (i) follows similarly to the argument in Lemma~\ref{wedgestar}; (ii) uses the fact that $\els$ is a right congruence. For  (iii),  let $b^* = (ab)^*$.
Then, using the ample condition,
\begin{equation*}
 b = bb^* = b(ab)^* = a^*b.
\end{equation*}

Conversely, let $b = a^*b$. Then, applying $^*$ to both sides,
\begin{equation*}
b^* = (a^*b)^* = (ab)^*,  
\end{equation*}
using (ii) in the last equality.
\end{proof}

\begin{lemma}\label{ampleL*}
Let $S$ be a right ample subsemigroup of an inverse semigroup $Q$.
Then $S$ is embedded as a unary semigroup into $Q$
if and only if $\LLL^Q \cap (S \times S) = \LLL^*$.
\end{lemma}

\begin{proof} For any $a,b\in S$ we have $a\,\els\, b$ if and only if $a^*=b^*$. Further, $a\,\el^Q\, b$ if and only if $a^{-1}a=b^{-1}b$. The  statement follows.  
\end{proof}

\subsection{Right ample straight left I-orders}\label{secrightample}

The aim of this subsection is to prove the following theorem.

\begin{theorem}\label{rightample}
Let $S$ be a right ample semigroup
and let $\RRR'$ be a binary relation on $S$.
Then $S$ has a semigroup of straight left I-quotients  $Q$,
such that $S$ is embedded in $Q$
as a unary semigroup
and $\RRR^Q \cap (S \times S) = \RRR'$,
if and only if $\RRR'$ is a left congruence on $S$
such that $S$ satisfies Conditions (A1) - (A3).
\begin{itemize}
\item[(A1)] For all $\alpha,\beta \in S$, there exists $\gamma,\delta \in S$ such that
$\gamma \, \RRR' \, \delta \, \RRR' \, \delta\beta = \gamma\alpha$
and $\alpha\beta^* = \gamma^*\alpha$.
\item[(A2)] For all $\alpha,\beta,\gamma \in S$,
$\gamma\alpha \, \RRR' \, \gamma\beta$ implies that
$\gamma^*\alpha \, \RRR' \, \gamma^*\beta$.
\item[(A3)] $\RRR' \subseteq \RRR^*$.
\end{itemize}
\end{theorem}

To prove this we will apply Theorem \ref{main}.
In order to apply Theorem \ref{main},
we must find $\leq_{\LLL^Q}$ and the associated meet.

\begin{lemma}\label{Right ample meet}
Let $S$ be a right ample semigroup embedded as a unary semigroup into an inverse semigroup $Q$.
Then for all $a,b,c,x \in S$,
\begin{enumerate}[(i)]
\item $a \leq_{\LLL^Q} b$
if and only if $a^* = a^*b^*$
if and only if $a \leq_{\LLL^*} b$;
\item $L_a \, \wedge \, L_b = L_{c} \,$
if and only if $\, c^* = a^*b^*$; and
\item $L_a \, \wedge \, L_b = L_{xa} \,$
if and only if $\, ab^* = x^*a$.
\end{enumerate}
\end{lemma}

\begin{proof} First recall that 
since $S$ is embedded in $Q$ in such a way that $^*$ is preserved,
we have that for all $a \in S$, $a^* = a^{-1}a$.

\begin{enumerate}[(i)]
\item By Lemma \ref{leqsame},
$a^* \leq b^*$, $a^* \leq_{\LLL^Q} b^*$
and $a^* \leq_{\LLL^*} b^*$ are all equivalent.
The result is then obtained by noticing that
$a \, \LLL^* \, a^*$ and
$a \, \LLL^Q \, a^*$.
\item Using (i), we can see that the poset $Q / \LLL$ is order isomorphic to
$Q / \LLL^*$.
Therefore $L_{a} \, \wedge \, L_{b} = L_{c} \,$ if and only if
$L^*_{a} \, \wedge \, L^*_{b} = L^*_{c} \,$,
which is true exactly when $L^*_{a^*} \, \wedge \, L^*_{b^*} = L^*_{c^*} \,$.
By Lemma \ref{wedgestar}, we know that $L^*_{a^*} \, \wedge \, L^*_{b^*} = L^*_{a^*b^*} \,$.
Therefore $L^*_{a^*} \, \wedge \, L^*_{b^*} = L^*_{c^*} \,$ if and only if $c^* = a^*b^*$.
\item Let $L_a \, \wedge \, L_b = L_{xa} \,$.
Therefore, by (ii), we have $a^*b^* = (xa)^*$.
Then, using Lemma \ref{right ample properties},
\begin{equation*}
ab^* = aa^*b^* = a(xa)^* = x^*a.  
\end{equation*}
On the other hand let $ab^* = x^*a$.
Then, again  using  Lemma \ref{right ample properties},
\begin{equation*}
(xa)^* = (x^*a)^* = (ab^*)^* = (a^*b^*)^* = a^*b^*,
\end{equation*}
and so by (ii), we have $L_a \, \wedge \, L_b = L_{xa} \,$.
\end{enumerate}
\end{proof}

We now prove Theorem \ref{rightample}.

\begin{proof}
We consider the forward implication first.
Let $S$ be a right ample straight left I-order in $Q$,
such that $S$ is embedded in $Q$
as a unary semigroup,
and let ${\RRR' = \RRR^Q \cap (S \times S)}$.
We know that $\RRR'$ is a left congruence and that therefore
$\gamma \alpha \, \RRR' \, \gamma \beta$ implies that
$\gamma^{-1}\gamma \alpha \, \RRR' \, \gamma^{-1}\gamma \beta$,
so (A2) is satisfied.
Using Theorem \ref{main}, we know that (M1) and (M4) are satisfied,
which are exactly (A1) and (A3) respectively,
using Lemma \ref{Right ample meet} (iii).

We will prove the converse
by proving each property in Theorem \ref{main} with
${\leq_l \, = \, \leq_{\LLL^*}}$, so that
$\el'=\els$.
Using the fact that $a \, \LLL^* \, a^*$ for all $a \in S$
along with Lemma \ref{leqsame}, we see that
$a \leq_l b$ if and only if $a^*b^* = a^*$,
for $a,b, \in S$. 
We already know that $\leq_{\LLL^*}$ is a right compatible preorder.
By Lemma \ref{wedgestar}, we know that
$L^*_{a^*} \, \wedge \, L^*_{b^*} = L^*_{a^*b^*}$.
Therefore, using the fact that
there is a unique idempotent in each $\LLL^*$-class,
we have that
\begin{equation*}
L^*_{a} \, \wedge \, L^*_{b} = L^*_{c}
\, \text{ if and only if } \, c^* = a^*b^*.
\end{equation*}

We will now prove (M1) - (M6) with $\leq_l \, = \, \leq_{\LLL^*}$
in order to satisfy the conditions of Theorem \ref{main}.

\begin{enumerate}
\item[(M1)] Let $\alpha,\beta \in S$.
Applying Property (A1), there exists $\gamma,\delta \in S$ such that
$\gamma \, \RRR' \, \delta \, \RRR' \, \delta\beta = \gamma\alpha$
and $\alpha\beta^* = \gamma^*\alpha$.
We can use Lemma \ref{right ample properties} (ii),
along with $\alpha\beta^* = \gamma^*\alpha$ to obtain
\begin{equation*}
(\gamma\alpha)^* = (\gamma^*\alpha)^* = (\alpha\beta^*)^*
= (\alpha^*\beta^*)^* = \alpha^*\beta^*.
\end{equation*}
Therefore $L^*_{\alpha} \, \wedge \, L^*_{\beta} = L^*_{\gamma\alpha}$.

\item[(M3)] By definition, we know that $\alpha \beta \leq_{\LLL^*} \beta$.

\item[(M2)] Let $L^*_{\alpha} \, \wedge \, L^*_{\beta} = L^*_{\gamma} \,$.
Then $\gamma^* = \alpha^*\beta^*$. Also, let $\delta \in S$.
We use the ample condition twice, to get
\begin{equation*}
{\delta(\alpha \delta)^*(\beta \delta)^* = \alpha^*\delta(\beta \delta)^* = \alpha^*\beta^*\delta}.
\end{equation*}
Therefore, using Lemma \ref{right ample properties} (ii),
\begin{equation}\label{RA1}
(\alpha^*\beta^*\delta)^* = (\delta(\alpha \delta)^*(\beta \delta)^*)^* =
(\delta^*(\alpha \delta)^*(\beta \delta)^*)^* = \delta^*(\alpha \delta)^*(\beta \delta)^*.
\end{equation}
Also, since $\alpha \delta \leq_l \delta$ by (M3),
we have that
\begin{equation}\label{RA2}
\delta^*(\alpha \delta)^* = (\alpha \delta)^*.
\end{equation}
Lastly, using Lemma \ref{right ample properties} (ii),
\begin{equation}\label{RA3}
(\gamma \delta)^* = (\gamma^*\delta)^* = (\alpha ^*\beta^*\delta)^*.
\end{equation}
Putting all this together,
\begin{equation*}
(\gamma \delta)^* \overset{(\ref{RA3})}{=} (\alpha^*\beta^*\delta)^* \overset{(\ref{RA1})}{=}
\delta^*(\alpha \delta)^*(\beta \delta)^* \overset{(\ref{RA2})}{=} (\alpha \delta)^*(\beta \delta)^*,
\end{equation*}
which gives us that
$L^*_{\alpha \delta} \, \wedge \, L^*_{\beta \delta} = L^*_{\gamma \delta} \,$.

\item[(M4)] This is Property (A3).

\item[(M5)] Let $\gamma \, \RRR' \, \gamma\alpha \, \LLL^* \, \alpha$ and let
$\delta \, \RRR' \, \delta\beta \, \LLL^* \, \beta$.
We have that $\gamma\gamma^* = \gamma \, \RRR' \, \gamma\alpha$.
Therefore we can use Property (A2) to obtain
$\gamma^* = \gamma^*\gamma^* \, \RRR' \, \gamma^*\alpha$.
We also have that $\gamma\alpha \, \LLL^* \, \alpha$, and so
$(\gamma\alpha)^* = \alpha^*$.
By Lemma \ref{right ample properties} (iii)
this is equivalent to ${\alpha = \gamma^*\alpha}$.
Similarly $\delta^* \, \RRR' \, \delta^*\beta$ and ${\beta = \delta^*\beta}$.

Let $\gamma \, \LLL^* \, \delta$, and so $\gamma^* = \delta^*$.
Therefore $\alpha = \gamma^*\alpha \, \RRR' \, \gamma^* = \delta^* \, \RRR' \, \delta^*\beta = \beta$.

Conversely, let $\alpha \, \RRR' \, \beta$.
We see that $\gamma^* \, \RRR' \, \gamma^*\alpha =
\alpha \, \RRR' \, \beta = \delta^*\beta \, \RRR' \, \delta^*$,
and therefore using (A3), $\gamma^* \, \RRR^* \, \delta^*$.
We know that since $E(S)$ is a semilattice, there can only be one
idempotent in each \mbox{$\RRR^*$-class}, and so $\gamma^* = \delta^*$.

\item[(M6)] Let $\alpha \, \LLL^* \, \beta
\, \LLL^* \, \gamma\alpha \, \LLL^* \, \gamma\beta$ and let $\gamma\alpha = \gamma\beta$.
We have that
\begin{equation*}
\alpha^* = \beta^* = (\gamma\alpha)^* = (\gamma\beta)^*, 
\end{equation*}
and so we can use \mbox{Lemma \ref{right ample properties} (iii)} to give us that
$\alpha = \gamma^*\alpha$ and $\beta = \gamma^*\beta$.
We then use the fact that $\gamma \, \LLL^* \, \gamma^*$, to give us that
$\gamma\alpha = \gamma\beta$ implies that $\gamma^*\alpha = \gamma^*\beta$,
and therefore $\alpha = \beta$.
\end{enumerate}

Therefore, $S$ with $\, \leq_l \, =  \, \leq_{\LLL^*}$
satisfies the conditions of
Theorem \ref{main} and we can apply
Theorem \ref{main} to give us that
$S$ has a semigroup of straight left \mbox{I-quotients}, $Q$,
such that $\RRR^Q \cap (S \times S) = \RRR'$
and $\leq_{\LLL^Q} \cap \, (S \times S) = \, \leq_{\LLL^*}$.
Therefore $\LLL^Q \cap (S \times S) = \LLL^*$,
and so by Lemma \ref{ampleL*}, $S$ is embedded in $Q$ as a unary semigroup.
\end{proof}

\subsection{Two-sided ample left I-orders} \label{sectwosidedample}

Finally we consider (two-sided) ample semigroups as left I-orders.
Here we have a pleasing description much more reminiscent of the Ore result describing classical left orders in groups. 

\begin{corollary}\label{twosidedample}
Let $S$ be an  ample semigroup.
Then $S$ has a semigroup of left I-quotients $Q$ such that $S$ is a bi-unary subsemigroup of $Q$, 
if and only if
for all $b,c \in S$,
there exists $u,v \in S$ such that
\begin{equation}
ub = vc, \; u^+ = v^+ = (vc)^+ \mbox{ and } bc^* = u^*b. \tag{$\star$}
\end{equation}
\end{corollary}

\begin{proof}
We first consider the forward implication.
Let $S$ be bi-unary subsemigroup of an inverse semigroup $Q$ such that $Q$ is a 
semigroup of left I-quotients of $S$. 
We know that $a \, \RRR^Q \, b$ if and only if $a^+ = b^+$
and $a \, \LLL^Q \, b$ if and only if $a^* = b^*$.
By Lemma 2.4 of \cite{ghroda2012semigroups},
we know that $S$ is straight in $Q$.
Therefore, by Theorem \ref{rightample}, Property (A1) is satisfied.
Therefore ($\star$) is satisfied.

For the backward implication,
we aim to apply Theorem \ref{rightample} with $\RRR' = \RRR^*$.
That is, $a \, \RRR' \, b$ if and only if
$a^+ = b^+$. Note that  $\ars$  is a left congruence.
We now prove Properties (A1) - (A3).

\begin{itemize}
\item[(A1)] Satisfied since  ($\star$) is assumed.

\item[(A2)] Let $xa \, \RRR^* \, xb$. This means
\begin{equation*}
(xa)^+ = (xb)^+.
\end{equation*}
We apply the dual of Lemma \ref{right ample properties} (ii) to get
\begin{equation*}
(xa^+)^+ = (xb^+)^+. 
\end{equation*}
Right multiplying this by $x$ gives us
\begin{equation*}
(xa^+)^+x = (xb^+)^+x.
\end{equation*}
We then apply the left ample property to give us 
\begin{equation*}
xa^+ = xb^+.
\end{equation*}
By the definition of $\LLL^*$, we know that $x \, \LLL^* \, x^*$.
Therefore, the above equation implies that
\begin{equation*}
x^*a^+ = x^*b^+.
\end{equation*}
Therefore, applying $^+$ to both sides, we have
\begin{equation*}
(x^*a^+)^+ = (x^*b^+)^+.
\end{equation*}
Then, the dual of Lemma \ref{right ample properties} (ii) gives us
\begin{equation*}
(x^*a)^+ = (x^*b)^+,
\end{equation*}
and so $x^*a \, \RRR^* \, x^*b$.

\item[(A3)] $\RRR' = \RRR^*$.
\end{itemize}

Therefore, Theorem \ref{rightample} gives us that
$S$ has a straight left I-order $Q$ such that
$^*$ is preserved
and $\RRR^Q \cap (S \times S) = \RRR^*$.
Therefore, by the dual of Lemma \ref{ampleL*}, $^+$ is also preserved.
\end{proof}

\section*{Acknowledgement and Open Question} The authors would like to thank Mark Kambites for his suggestions which led to Example \ref{nonstraight}. We remark again that many of the left I-orders that naturally occur are straight. We end by posing the following question: is every left I-order in a proper inverse semigroup straight?
 
\bibliographystyle{abbrv}
\bibliography{references.bib}{}

\end{document}